\documentclass[final,3p,times]{elsarticle}
\usepackage{multicol,amssymb,amsmath,amsbsy,amsthm,graphicx,esvect}
\usepackage[latin1]{inputenc}
\usepackage{caption2}

\theoremstyle{plain}
\newtheorem{theo}{Theorem}[section]
\newtheorem{rmk}[theo]{Remark}
\newtheorem{lem}[theo]{Lemma}
\newtheorem{co}[theo]{Corollary}
\newtheorem{prop}[theo]{Proposition}

\theoremstyle{definition}

\newcommand{\supp}{\text{\rm{supp}}}
\newcommand{\sign}{\text{\rm{sign}}}
\newcommand{\Reg}{\text{\rm{\textbf{Reg}}}}
\newcommand{\capp}{\text{\rm{cap}}}
\newcommand{\Co}{\text{\rm{Co}}}

\begin{document}

\begin{frontmatter}



\title{Asymptotics of multiple orthogonal polynomials for a
system of two measures supported on a starlike set}


\author{A. L\'{o}pez Garc\'{i}a \corref{cor1}}\cortext[cor1]{Current address: Department of Mathematics, Katholieke
Universiteit Leuven, Celestijnenlaan 200B, B-3001, Leuven,
Belgium. Telephone number: +32 16 327053. Fax number: +32 16
327998.}

\ead{abey.lopezgarcia@wis.kuleuven.be}

\address{Department of Mathematics, Vanderbilt University, Nashville, TN 37240, USA}

\begin{abstract}
For a system of two measures supported on a starlike set in the
complex plane, we study asymptotic properties of associated
multiple orthogonal polynomials $Q_{n}$ and their recurrence
coefficients. These measures are assumed to form a Nikishin-type
system, and the polynomials $Q_{n}$ satisfy a three-term
recurrence relation of order three with positive coefficients.
Under certain assumptions on the orthogonality measures, we prove
that the sequence of ratios $\{Q_{n+1}/Q_{n}\}$ has four different
periodic limits, and we describe these limits in terms of a
conformal representation of a compact Riemann surface. Several
relations are found involving these limiting functions and the
limiting values of the recurrence coefficients. We also study the
$n$th root asymptotic behavior and zero asymptotic distribution of
$Q_{n}$.
\end{abstract}

\begin{keyword}
Higher-order recurrences \sep Nikishin systems \sep ratio
asymptotics \sep $n$th root asymptotics \sep zero asymptotic
distribution.

\MSC 30C15 \sep 31A15 \sep 30E25.
\end{keyword}

\end{frontmatter}


\section{Introduction and statement of main results}
\label{sectionprimeraMOP}

This work was motivated by recent investigations of Aptekarev et
al. \cite{AptKalSaff} on asymptotic properties of monic
polynomials $Q_{n}$ generated by the higher-order three-term
recurrence relation
\begin{equation}\label{higherorderrecurrence}
z Q_{n}=Q_{n+1}+a_{n}Q_{n-p},\qquad n\geq p,\quad
p\in\mathbb{N},\quad a_{n}>0,
\end{equation}
with initial conditions
\begin{equation}\label{initialconditionsAKS}
Q_{j}(z)=z^{j},\qquad j=0,\ldots,p.
\end{equation}
In \cite{AptKalSaff}, strong asymptotics of $Q_{n}$ was studied
assuming that the recurrence coefficients satisfy
\begin{equation}\label{perturbationcond}
\sum_{n=p}^{\infty}|a_{n}-a|<\infty, \qquad a>0.
\end{equation}

An important element in the asymptotic analysis of the polynomials
$Q_{n}$ is the starlike set
\[
\widetilde{S}_{0}:=\bigcup_{k=0}^{p}\exp(2\pi i
k/(p+1))\,[0,\alpha], \qquad
\alpha:=[(p+1)/p^{p/(p+1)}]a^{1/(p+1)}.
\]
In fact, \cite[Theorem 7.2]{AptKalSaff} asserts that
\[
\lim_{n\rightarrow\infty}\frac{Q_{n}(z)}{w_{0}^{n}(z)}=F_{0}(z),
\qquad\mbox{uniformly on compact subsets of}\quad
\overline{\mathbb{C}}\setminus \widetilde{S}_{0},
\]
where $w_{0}(z)$ is the unique branch of the algebraic equation
$w^{p+1}-z w^{p}+a=0$ that is meromorphic at infinity and has an
analytic continuation in $\mathbb{C}\setminus\widetilde{S}_{0}$.

We remark that notable families of polynomials satisfy
(\ref{higherorderrecurrence}) in the constant coefficients case,
for example the classical monic Tchebyshev polynomials
$T_{n}(x)=2\cos n(\cos^{-1}(x/2))$ for the segment $[-2,2]$ ($p=1,
a_{n}=1$ for all $n$). More generally, it was shown by He and Saff
\cite{HeSaff} that the Faber polynomials associated with the
closed domain bounded by the $(p+1)$-cusped hypocycloid with
parametric equation
\[
z=\exp(i\,\theta)+\frac{1}{p} \exp(-p\, i\, \theta),\qquad 0\leq
\theta<2\pi,\quad p\geq 2,
\]
are also generated by the recurrence relation
(\ref{higherorderrecurrence}) with constant coefficients
$a_{n}=a=1/p$, and their zeros are contained in
$\widetilde{S}_{0}$. Many other properties of the zeros of these
Faber polynomials were obtained in \cite{HeSaff} and \cite{EV}.

Using operator theoretic techniques, in \cite{AptKalVanI} it was
proved that the polynomials $Q_{n}$ generated by
(\ref{higherorderrecurrence})--(\ref{initialconditionsAKS}) are in
fact multiple orthogonal polynomials with respect to a system of
$p$ measures supported on
\[
\bigcup_{k=0}^{p}\exp(2\pi i k/(p+1))\,[0,\infty).
\]
Moreover, if (\ref{perturbationcond}) holds then the orthogonality
measures have a specific hierarchy structure; they form a
Nikishin-type system (see Section 8 and Theorem 9.1 in
\cite{AptKalSaff}). This system is the system of spectral measures
of the banded Hessenberg operator (with only two nonzero
diagonals) associated with (\ref{higherorderrecurrence}).

In this paper we study, among other topics, ratio and $n$th root
asymptotics of multiple orthogonal polynomials associated with a
Nikishin-type system of two measures (p=2) supported on a star,
starting from assumptions on these measures of orthogonality. For
simplicity we assume that these measures are given by weights.
Under similar assumptions, analogous results can be obtained for
general measures. We start with the definition of the
orthogonality weights and the associated polynomials.

Let
\[
S_{0}:=\bigcup_{k=0}^{2}\exp(2\pi i k/3)\,[0,\alpha],\qquad
0<\alpha<\infty.
\]
We emphasize that $\alpha$ is here arbitrary. Assume that $s_{1}$
is a complex-valued function defined on $S_{0}$, such that
\[
s_{1}\geq 0 \quad \mbox{on} \quad (0,\alpha),\qquad s_{1}\in
L^{1}(0,\alpha),
\]
\begin{equation}\label{eq:symsigma1}
s_{1}(e^{\frac{2\pi i}{3}}z)=e^{\frac{4\pi i}{3}}s_{1}(z),\quad
z\in S_{0}\setminus\{0, \alpha, e^{\frac{2\pi i}{3}}\alpha,
e^{\frac{4\pi i}{3}}\alpha\}.
\end{equation}
Set
\[
f(z):=z^2\int_{-b}^{-a}\frac{s_{2}(t)}{z^{3}-t^{3}}\,dt,\qquad
0<a<b<\infty,
\]
where $s_{2}$ is a non-negative integrable function defined on
$[-b,-a]$. We assume of course that $s_{1}(t)\,dt,
s_{2}(t)\,dt\neq 0$.

Let $\{Q_{n}\}_{n=0}^{\infty}$ be the sequence of \textit{monic}
polynomials of lowest degree that satisfy the following
conditions:
\begin{equation}\label{eq:ortQn1}
\int_{S_{0}}Q_{2n}(t)\,t^{k}\,s_{1}(t)\,dt=0,\qquad
k=0,\ldots,n-1,
\end{equation}
\begin{equation}\label{eq:ortQnn}
\int_{S_{0}}Q_{2n}(t)\,t^{k}\,f(t)\,s_{1}(t)\,dt=0,\qquad
k=0,\ldots,n-1,
\end{equation}
\begin{equation}\label{eq:ortQn2}
\int_{S_{0}}Q_{2n+1}(t)\,t^{k}\,s_{1}(t)\,dt=0,\qquad
k=0,\ldots,n,
\end{equation}
\begin{equation}\label{eq:ortQn3}
\int_{S_{0}}Q_{2n+1}(t)\,t^{k}\,f(t)\,s_{1}(t)\,dt=0,\qquad
k=0,\ldots,n-1.
\end{equation}

These are the polynomials whose algebraic and asymptotic
properties we investigate.

\begin{prop}\label{propmaximal}
The degree of each polynomial $Q_{n}$ is maximal, i.e.
$\deg{Q_{n}}=n$. Moreover, if $n=3j$, then $Q_{n}$ has exactly $j$
simple zeros on the interval $(0,\alpha)$. If $n=3j+1$, then
$Q_{n}$ has a simple zero at the origin and $j$ simple zeros on
$(0,\alpha)$. Finally, if $n=3j+2$, then $Q_{n}$ has a double zero
at the origin and $j$ simple zeros on $(0,\alpha)$. The remaining
zeros of $Q_{n}$ are located on the rays $\exp(2\pi i
/3)\,(0,\alpha)$, $\exp(4\pi i /3)\,(0,\alpha)$, and are rotations
of the zeros on $(0,\alpha)$.
\end{prop}



\begin{prop}\label{proprecurrence}
The monic polynomials $Q_{n}$ satisfy the following three-term
recurrence relation
\begin{equation}\label{recurrence}
z Q_{n}=Q_{n+1}+a_{n}Q_{n-2},\qquad n\geq 2,\quad
a_{n}\in\mathbb{R},
\end{equation}
where
\begin{equation}\label{initialcond}
Q_{j}(z)=z^{j},\quad j=0,1,2.
\end{equation}
The coefficients $a_{n}$ are given by the formulas
\begin{equation}\label{intrepa2n}
a_{2n}=\frac{\int_{0}^{\alpha}t^{n}\,Q_{2n}(t)\,s_{1}(t)\,dt}
{\int_{0}^{\alpha}t^{n-1}\,Q_{2n-2}(t)\,s_{1}(t)\,dt},\qquad\qquad
a_{2n+1}=\frac{\int_{0}^{\alpha}t^{n}\,Q_{2n+1}(t)\,f(t)\,s_{1}(t)\,dt}
{\int_{0}^{\alpha}t^{n-1}\,Q_{2n-1}(t)\,f(t)\,s_{1}(t)\,dt}.
\end{equation}
Moreover, $a_{n}>0$ for all $n\geq 2$.
\end{prop}

Propositions \ref{propmaximal} and \ref{proprecurrence} are proved
in Section \ref{sectionQnandzeros}. Let
\begin{equation*}\label{defnPsin1}
\Psi_{n}(z):=\int_{S_{0}}\frac{Q_{n}(t)}{t-z}\,s_{1}(t)\,dt.
\end{equation*}
The functions $\Psi_{n}$ (usually called \textit{functions of
second type}) satisfy:
\begin{equation}\label{eq:BVPPsin1}
\left\{
\begin{array}{ll}
\Psi_{n}\in H(\overline{\mathbb{C}}\setminus S_{0}),\\
\\
\Psi_{2n}(z)=O(1/z^{n+1}),& z\rightarrow\infty,\\
\\
\Psi_{2n+1}(z)=O(1/z^{n+2}),& z\rightarrow\infty.
\end{array}
\right.
\end{equation}
It is important for our analysis to determine the exact number of
zeros of $\Psi_{n}$ outside $S_{0}$, and their location. The
following result, proved in Section \ref{sectionsecondtypeQn2},
gives the answers to these questions.

\begin{prop}\label{propzerosPsi1}
For each $j\in\{0,1,2,3,5\}$, the function $\Psi_{6l+j}$ has
exactly $3l$ simple zeros in $\mathbb{C}\setminus S_{0}$, of which
$l$ zeros are located in $(-b,-a)$, and the remaining $2l$ zeros
are rotations of these $l$ zeros by angles of $2\pi/3$ and
$4\pi/3$; $\Psi_{6l+j}$ has no other zeros in $\mathbb{C}\setminus
S_{0}$. The function $\Psi_{6l+4}$ has exactly $3l+3$ simple zeros
in $\mathbb{C}\setminus S_{0}$, of which $l+1$ zeros are located
in $(-b,-a)$, and the remaining $2l+2$ zeros are rotations of
these $l+1$ zeros by angles of $2\pi/3$ and $4\pi/3$;
$\Psi_{6l+4}$ has no other zeros in $\mathbb{C}\setminus S_{0}$.
\end{prop}

\textbf{Notation:} Let $Q_{n,2}$ denote the \textit{monic}
polynomial whose zeros coincide with the zeros of $\Psi_{n}$ in
$\mathbb{C}\setminus S_{0}$.

The following result asserts that for consecutive values of $n$,
the zeros of $Q_{n}$ interlace, and the same is true for the zeros
of $Q_{n,2}$.

\begin{theo}\label{theointerlacing}
For every $n\geq 0$, the polynomials $Q_{n}$ and $Q_{n+1}$ do not
have any common zeros in $S_{0}\setminus\{0\}$. Moreover, there is
exactly one zero of $Q_{n+1}$ between two consecutive zeros of
$Q_{n}$ in $(0,\alpha)$. Conversely, there is exactly one zero of
$Q_{n}$ between two consecutive zeros of $Q_{n+1}$ in
$(0,\alpha)$.

Additionally, for every $n\geq 0$, the functions $\Psi_{n}$ and
$\Psi_{n+1}$ do not have any common zeros in $S_{1}$. There is
exactly one zero of $\Psi_{n+1}$ between two consecutive zeros of
$\Psi_{n}$ in $(-b,-a)$, and vice versa.
\end{theo}

Theorem \ref{theointerlacing} is proved in Section
\ref{sectioninterlacing}. We can determine exactly how the zeros
of $Q_{n}$ interlace, thanks to the fact that the recurrence
coefficients $a_{n}$ are all positive (see Proposition
\ref{interlacingzerosQn} in Section \ref{sectioninterlacing}).



We next describe the ratio asymptotics of the polynomials $Q_{n}$
and $Q_{n,2}$, and the limiting behavior of the recurrence
coefficients $a_{n}$. By Propositions \ref{propmaximal} and
\ref{propzerosPsi1}, for some polynomials $P_{n}$ and $P_{n,2}$ we
may write:
\begin{equation}\label{defn:P3k}
Q_{3k}(\tau)=P_{3k}(\tau^3),\qquad Q_{3k+1}(\tau)=\tau
P_{3k+1}(\tau^3),\qquad Q_{3k+2}(\tau)=\tau^2 P_{3k+2}(\tau^3),
\end{equation}
\begin{equation}\label{defn:Pn2}
Q_{n,2}(\tau)=P_{n,2}(\tau^3).
\end{equation}
Let
\begin{equation}\label{defnSrho}
S_{1}:=\bigcup_{k=0}^{2}\exp(2\pi i k/3)\,[-b,-a].
\end{equation}

\begin{theo}\label{introd:ratioasymptotics}
Assume that $s_{1}>0$ a.e. on $[0,\alpha]$ and $s_{2}>0$ a.e. on
$[-b,-a]$. Then, for each $i\in\{0,\ldots,5\}$, the following
limits hold:
\begin{equation}\label{eq:ratioasymp1}
\lim_{k\rightarrow\infty}\frac{P_{6k+i+1}(z)}{P_{6k+i}(z)}
=\widetilde{F}_{1}^{(i)}(z),\qquad
z\in\mathbb{C}\setminus[0,\alpha^{3}],
\end{equation}
\begin{equation}\label{eq:ratioasymp2}
\lim_{k\rightarrow\infty}\frac{P_{6k+i+1,2}(z)}{P_{6k+i,2}(z)}
=\widetilde{F}_{2}^{(i)}(z),\qquad
z\in\mathbb{C}\setminus[-a^{3},-b^{3}],
\end{equation}
where convergence is uniform on compact subsets of the indicated
regions. Moreover $($cf. $($$\ref{recurrence}$$)$$)$,
\begin{equation}\label{ratiocoeff}
\lim_{k\rightarrow\infty}a_{6k+i}=\left\{
\begin{array}{lll}
-C_{1}^{(i)}, &  \mbox{for} &i\in\{0,1,3,4\},\\
\\
-C_{0}^{(i)}, & \mbox{for} &i\in\{2,5\},\\
\end{array}
\right.
\end{equation}
where
\begin{equation}\label{Laurentexpansion}
\widetilde{F}_{1}^{(i)}(z)=\left\{
\begin{array}{lll}
1+C_{1}^{(i)}/z+O(1/z^2), &  \mbox{for} &i\in\{0,1,3,4\},\\
\\
z+C_{0}^{(i)}+O(1/z), &  \mbox{for} &i\in\{2,5\},\\
\end{array}
\right.
\end{equation}
is the Laurent expansion at $\infty$ of $\widetilde{F}_{1}^{(i)}$.
Consequently, the limits
\begin{equation}\label{eq:ratioQ1}
\lim_{k\rightarrow\infty}\frac{Q_{6k+i+1}(z)}{Q_{6k+i}(z)}=z\,\widetilde{F}_{1}^{(i)}(z^{3}),\qquad
z\in\mathbb{C}\setminus S_{0},\quad i\in\{0,1,3,4\},
\end{equation}
\begin{equation}\label{eq:ratioQ2}
\lim_{k\rightarrow\infty}\frac{Q_{6k+i+1}(z)}{Q_{6k+i}(z)}=\frac{\widetilde{F}_{1}^{(i)}(z^{3})}{z^2},\qquad
z\in\mathbb{C}\setminus S_{0},\quad i\in\{2,5\},
\end{equation}
\begin{equation}\label{eq:ratioQ3}
\lim_{k\rightarrow\infty}\frac{Q_{6k+i+1,2}(z)}{Q_{6k+i,2}(z)}
=\widetilde{F}_{2}^{(i)}(z^{3}),\qquad z\in\mathbb{C}\setminus
S_{1},\quad i\in\{0,\ldots,5\},
\end{equation}
hold uniformly on compact subsets of the indicated regions.
\end{theo}

We also describe in Proposition \ref{propratioasymporthonPsi}
(Section \ref{sectionratio}) the ratio asymptotic behavior of the
functions of second type $\Psi_{n}$, as well as the ratio
asymptotic behavior of the polynomials $p_{n}, p_{n,2}$ defined in
$(\ref{eq:definitionpn})$ (these polynomials are ``orthonormal
versions" of the polynomials $P_{n}, P_{n,2}$ defined in
(\ref{defn:P3k})--(\ref{defn:Pn2}), see Proposition
\ref{orthonormalitypn}) and their leading coefficients.

Several relations can be established among the limiting functions
$\widetilde{F}_{1}^{(i)}, \widetilde{F}_{2}^{(i)}$, and the
limiting values of the recurrence coefficients (see also the
boundary value properties described in Proposition \ref{prop5}).

Let us define
\[
a^{(i)}:=\lim_{k\rightarrow\infty}a_{6k+i},\qquad 0\leq i\leq 5.
\]

\begin{prop}\label{introd:relation}
The following relations among the functions
$\widetilde{F}_{j}^{(i)}$ are valid:
\begin{equation}\label{eq:relfunc11}
\widetilde{F}_{1}^{(2)}(z)=z\widetilde{F}_{1}^{(0)}(z),\qquad
\widetilde{F}_{1}^{(5)}(z)=z\widetilde{F}_{1}^{(3)}(z),
\end{equation}
\begin{equation}\label{eq:relfunc12}
\widetilde{F}_{1}^{(0)}\widetilde{F}_{1}^{(1)}=\widetilde{F}_{1}^{(3)}\widetilde{F}_{1}^{(4)},
\qquad\widetilde{F}_{1}^{(1)}\widetilde{F}_{1}^{(2)}=\widetilde{F}_{1}^{(4)}\widetilde{F}_{1}^{(5)},
\qquad\widetilde{F}_{1}^{(2)}\widetilde{F}_{1}^{(3)}=\widetilde{F}_{1}^{(5)}\widetilde{F}_{1}^{(0)},
\end{equation}
\begin{equation}\label{eq:relfunc13}
\frac{1-\widetilde{F}_{1}^{(3)}}{1-\widetilde{F}_{1}^{(0)}}
=\frac{a^{(3)}}{a^{(0)}},\qquad
\frac{1-\widetilde{F}_{1}^{(4)}}{1-\widetilde{F}_{1}^{(1)}}
=\frac{a^{(4)}}{a^{(1)}},\qquad
\frac{z-\widetilde{F}_{1}^{(5)}(z)}{z-\widetilde{F}_{1}^{(2)}(z)}
=\frac{a^{(5)}}{a^{(2)}},
\end{equation}
\begin{equation}\label{eq:relfunc21}
\widetilde{F}_{2}^{(0)}=\widetilde{F}_{2}^{(2)},\qquad
\widetilde{F}_{2}^{(3)}=\widetilde{F}_{2}^{(5)},
\end{equation}
\begin{equation}\label{eq:relfunc22}
\widetilde{F}_{2}^{(0)}\widetilde{F}_{2}^{(1)}=\widetilde{F}_{2}^{(3)}\widetilde{F}_{2}^{(4)},
\qquad\widetilde{F}_{2}^{(1)}\widetilde{F}_{2}^{(2)}=\widetilde{F}_{2}^{(4)}\widetilde{F}_{2}^{(5)},
\qquad\widetilde{F}_{2}^{(2)}\widetilde{F}_{2}^{(3)}=\widetilde{F}_{2}^{(5)}\widetilde{F}_{2}^{(0)}.
\end{equation}
Furthermore, the functions $\widetilde{F}_{1}^{(i)}$,
$i\in\{0,\ldots,5\}$, are all distinct, and the functions
$\widetilde{F}_{2}^{(i)}$, $i\in\{0,1,3,4\}$, are also distinct.

For every $i\in\{0,\ldots,5\}$, $a^{(i)}>0$, and the following
relations hold:
\begin{equation}\label{eq:rellimitcoeff}
a^{(0)}=a^{(2)},\qquad a^{(3)}=a^{(5)},\qquad
a^{(0)}+a^{(1)}=a^{(3)}+a^{(4)}.
\end{equation}
The following inequalities also hold:
\[
a^{(0)}\neq a^{(3)},\quad a^{(0)}\neq a^{(4)},\quad a^{(1)}\neq
a^{(3)},\quad a^{(1)}\neq a^{(4)}.
\]
\end{prop}

In fact, we will show that $a^{(4)}>a^{(1)}$, and therefore
(\ref{eq:rellimitcoeff}) implies $a^{(0)}>a^{(3)}$ (see Remark
\ref{remarksobreaes}). Theorem \ref{introd:ratioasymptotics} and
Proposition \ref{introd:relation} are proved in Section
\ref{sectionratio}.

We next describe the limiting functions $\widetilde{F}_{j}^{(i)}$
in terms of a conformal representation of a compact Riemann
surface. Let $\Delta_{1}:=[0,\alpha^3]$, and
$\Delta_{2}:=[-b^3,-a^3]$. Consider the three-sheeted Riemann
surface
\[
\mathcal{R}=\overline{\mathcal{R}_{0}\cup \mathcal{R}_{1}\cup
\mathcal{R}_{2}},
\]
formed by the consecutively ``glued" sheets
\begin{equation}\label{definitionRiemannsurfaceR}
\mathcal{R}_{0}:=\overline{\mathbb{C}}\setminus\Delta_{1}, \qquad
\mathcal{R}_{1}:=\overline{\mathbb{C}}\setminus(\Delta_{1}\cup\Delta_{2}),
\qquad \mathcal{R}_{2}:=\overline{\mathbb{C}}\setminus\Delta_{2}.
\end{equation}
Since $\mathcal{R}$ has genus zero, there exists a unique
conformal representation $\psi$ of $\mathcal{R}$ onto
$\overline{\mathbb{C}}$ satisfying:
\begin{equation}\label{eq:divisorcond1}
\psi(z)=-2 z/a^3+O(1),\qquad
z\rightarrow\infty^{(1)}\in\mathcal{R}_{1},
\end{equation}
\begin{equation}\label{eq:divisorcond2}
\psi(z)=B/z+O(1/z^2),\qquad
z\rightarrow\infty^{(2)}\in\mathcal{R}_{2},\quad B\neq 0.
\end{equation}
Here $-a^3$ is the right endpoint of $\Delta_{2}$. Let
$\{\psi_{k}\}_{k=0}^2$ denote the branches of $\psi$.

Finally, given an arbitrary function $H(z)$ that has in a
neighborhood of infinity a Laurent expansion of the form $H(z)=C
z^{k}+O(z^{k-1})$, $C\neq 0$, $k\in\mathbb{Z}$, we denote by
$\widetilde{H}$ the function $H/C$.

\begin{theo}\label{theoRiemannsurfacerepresentation}
The following representations are valid:
\[
\widetilde{F}_{1}^{(0)}
=\frac{a^{(0)}-a^{(3)}}{a^{(0)}\widetilde{\psi}_{0}-a^{(3)}},\qquad
\widetilde{F}_{1}^{(1)}=\frac{(a^{(4)}-a^{(1)})\,\widetilde{\psi}_{0}}
{a^{(4)}\widetilde{\psi}_{0}-a^{(1)}},\qquad
\widetilde{F}_{1}^{(2)}(z)=\frac{z(a^{(0)}-a^{(3)})}{a^{(0)}\widetilde{\psi}_{0}(z)-a^{(3)}},
\]
\[
\widetilde{F}_{1}^{(3)}=\frac{(a^{(0)}-a^{(3)})
\,\widetilde{\psi}_{0}}{a^{(0)}\widetilde{\psi}_{0}-a^{(3)}},
\qquad \widetilde{F}_{1}^{(4)}=\frac{a^{(4)}-a^{(1)}}
{a^{(4)}\widetilde{\psi}_{0}-a^{(1)}},\qquad
\widetilde{F}_{1}^{(5)}(z)=\frac{z(a^{(0)}-a^{(3)})
\,\widetilde{\psi}_{0}(z)}{a^{(0)}\widetilde{\psi}_{0}(z)-a^{(3)}},
\]
\[
\widetilde{F}_{2}^{(0)}(z)=\widetilde{F}_{2}^{(2)}(z)
=\frac{a^{(0)}(a^{(0)}-a^{(3)})\,z\,\widetilde{\psi}_{0}(z)\,\widetilde{\psi}_{2}(z)}{(a^{(0)}-a^{(3)}\omega_{1}^{(3)}
\widetilde{\psi}_{0}(z)\,\widetilde{\psi}_{2}(z)/\omega_{1}^{(0)})(a^{(0)}\widetilde{\psi}_{0}(z)-a^{(3)})},
\]
\[
\widetilde{F}_{2}^{(3)}(z)=\widetilde{F}_{2}^{(5)}(z)
=\frac{a^{(0)}(a^{(0)}-a^{(3)})\,z\,\widetilde{\psi}_{0}(z)}{(a^{(0)}-a^{(3)}\omega_{1}^{(3)}
\widetilde{\psi}_{0}(z)\,\widetilde{\psi}_{2}(z)/\omega_{1}^{(0)})(a^{(0)}\widetilde{\psi}_{0}(z)-a^{(3)})},
\]
\[
\widetilde{F}_{2}^{(1)}
=\frac{a^{(4)}-a^{(1)}}{\widetilde{\psi}_{2}(a^{(4)}\widetilde{\psi}_{0}-a^{(1)})
(\widetilde{\psi}_{1}-(\omega_{1}^{(1)}-1)/\omega_{1}^{(4)})},\qquad\widetilde{F}_{2}^{(4)}
=\frac{a^{(4)}-a^{(1)}}{(a^{(4)}\widetilde{\psi}_{0}-a^{(1)})
(\widetilde{\psi}_{1}-(\omega_{1}^{(1)}-1)/\omega_{1}^{(4)})}.
\]
The constants $\omega_{1}^{(l)}$ are the reciprocals of the
right-hand sides in the boundary value equations
$(\ref{eq:boundaryvaluetilde1})$--$(\ref{eq:boundaryvaluetilde3})$.
They can be written in terms of the limiting values $a^{(i)}$ as
follows:
\[
\omega_{1}^{(0)}=\omega_{1}^{(2)}=\frac{a^{(4)}-a^{(1)}}{a^{(0)}a^{(4)}},
\qquad\omega_{1}^{(3)}=\omega_{1}^{(5)}=\frac{a^{(0)}}{a^{(0)}-a^{(3)}},
\qquad\omega_{1}^{(1)}=\frac{a^{(4)}}{a^{(4)}-a^{(1)}},\qquad
\omega_{1}^{(4)}=\frac{a^{(0)}-a^{(3)}}{(a^{(0)})^2}.
\]
\end{theo}

Using Theorem 3.1 from \cite{LPRY}, we can easily describe the
cubic algebraic equation solved by $\psi$. The coefficients of
this equation can be computed exclusively in terms of the
endpoints of the intervals $\Delta_{1}$ and $\Delta_{2}$.

\begin{prop}\label{algebraicequation}
Let
\begin{equation}\label{defnlambdamu}
\lambda:=\frac{2b^3}{a^3}-1,\qquad \mu:=\frac{2\alpha^3}{a^3}+1,
\end{equation}
and let $\beta$ and $\gamma$ be the unique solutions of the
algebraic system
\[
\left\{
\begin{array}{l}
2(\beta+\gamma)(3-\beta\gamma-\beta-\gamma)(3-\beta\gamma+\beta+\gamma)
+(\lambda-\mu)(\beta-\gamma)^3=0,\\
(\lambda+\mu)^2(\beta-\gamma)^6
=4(3+\beta\gamma)^3(1-\beta\gamma)(2+\beta+\gamma)(2-\beta-\gamma),
\end{array}
\right.
\]
satisfying the conditions $-1<\gamma<\beta<1$. Then $w=\psi(z)$ is
the solution of the cubic equation
\begin{equation}\label{algebraicequationenunciado}
w^3+\Big[\frac{2z}{a^3}+1+\frac{3+h+\Theta_{2}-\Theta_{1}}{H(\beta)}\Big]w^2
+\Big[\frac{4z}{a^3H(\beta)}+\frac{2}{H(\beta)}+\frac{2+2h+\Theta_{2}-3\Theta_{1}}{H(\beta)^2}\Big]w
-\frac{2\Theta_{1}}{H(\beta)^3}=0,
\end{equation}
where
\[
H(z)=h+z+\frac{\Theta_{1}z}{1-z}+\frac{\Theta_{2}z}{1+z},\qquad
h=\frac{1}{4}(\beta+\gamma)\Big(2\beta\gamma-\frac{(\beta-\gamma)^2}{1-\beta\gamma}\Big),
\]
\[
\Theta_{1}=\frac{1}{4}(1-c)(1-d)(1-\beta)(1-\gamma),\qquad
\Theta_{2}=\frac{1}{4}(1+c)(1+d)(1+\beta)(1+\gamma),
\]
$c$ and $d$ are the solutions of the equation
\[
x^2+(\beta+\gamma)\,x+\frac{(\beta-\gamma)^2}{1-\beta\gamma}-3=0,
\]
satisfying $c<-1, d>1$.
\end{prop}

\begin{rmk}
Using $(\ref{algebraicequationenunciado})$ and Theorem
$\ref{theoRiemannsurfacerepresentation}$, it is easy to check that
\[
a^{(0)}-a^{(3)}=-\frac{a^{3}\Theta_{2}}{4\,
H(\beta)}=a^{(4)}-a^{(1)}.
\]
\end{rmk}

Theorem \ref{theoRiemannsurfacerepresentation} and Proposition
\ref{algebraicequation} are proved in Section
\ref{Riemannsurfacesection}. We now describe the results on $n$th
root asymptotics and zero asymptotic distribution for the
polynomials $Q_{n}$ and $Q_{n,2}$. First, we introduce certain
definitions and notations.

Given a compact set $E\subset\mathbb{C}$, let $\mathcal{M}_{1}(E)$
denote the space of all probability Borel measures supported on
$E$. If $P$ is a polynomial of degree $n$, we indicate by
$\mu_{P}$ the associated normalized zero counting measure, i.e.
\[
\mu_{P}:=\frac{1}{n}\sum_{P(x)=0}\delta_{x},
\]
where $\delta_{x}$ is the Dirac measure with unit mass at $x$ (in
the sum the zeros are repeated according to their multiplicity).
If $\mu\in\mathcal{M}_{1}(E)$, let
\[
V^{\mu}(z):=\int\log\frac{1}{|z-t|}\,d\mu(t),
\]
and for a sequence $\{\mu_{n}\}\subset\mathcal{M}_{1}(E)$,
$\mu_{n}\stackrel{*}{\longrightarrow}\mu$ refers to the
convergence of $\mu_{n}$ in the weak-star topology to $\mu$.

Let $E_{1}, E_{2}$ be compact subsets of $\mathbb{C}$, and let
$M=[c_{j,k}]$ be a real, positive definite, symmetric matrix of
order two. Given a vector measure
$\boldsymbol{\mu}=(\mu_{1},\mu_{2})
\in\mathcal{M}_{1}(E_{1})\times\mathcal{M}_{1}(E_{2})$, we define
the combined potential
\[
W^{\boldsymbol{\mu}}_{j}:=\sum_{k=1}^{2}c_{j,k}V^{\mu_{k}}, \qquad
j=1,2,
\]
and the constants
\[
\omega^{\boldsymbol{\mu}}_{j}:=\inf\{W_{j}^{\boldsymbol{\mu}}(x):
x\in E_{j}\}, \qquad j=1,2.
\]

It is well-known (see \cite[Chapter 5]{NikSor}) that if $E_{1},
E_{2}$ are regular with respect to the Dirichlet problem, and
$c_{j,k}\geq 0$ in case $E_{j}\cap E_{k}\neq \emptyset$, then
there exists a unique vector measure
$\overline{\boldsymbol{\mu}}=(\overline{\mu}_{1},\overline{\mu}_{2})
\in\mathcal{M}_{1}(E_{1}) \times\mathcal{M}_{1}(E_{2})$ satisfying
the properties
$W_{j}^{\overline{\boldsymbol{\mu}}}(x)=\omega_{j}^{\overline{\boldsymbol{\mu}}}$
for all $x\in\supp(\overline{\mu}_{j}), j=1,2$. The measure
$\overline{\boldsymbol{\mu}}$ is called the vector equilibrium
measure determined by the interaction matrix $M$ on the system of
compact sets $(E_{1}, E_{2})$, and
$\omega_{1}^{\overline{\boldsymbol{\mu}}},
\omega_{2}^{\overline{\boldsymbol{\mu}}}$ are called the
equilibrium constants.

Let $\lambda_{1}$ be the positive, rotationally invariant measure
on $S_{0}$ whose restriction to the interval $[0,\alpha]$
coincides with the measure $s_{1}(x)\,dx$, and let $\lambda_{2}$
be the positive, rotationally invariant measure on $S_{1}$ whose
restriction to the interval $[-b,-a]$ coincides with the measure
$s_{2}(x)\,dx$.

Let $\Reg$ denote the space of regular measures in the sense of
Stahl and Totik (see definition in \cite[pg. 61]{StahlTotik}). The
zero asymptotic distribution and $n$th root asymptotics for the
polynomials $P_{n}$ and $P_{n,2}$ can be described as follows:

\begin{theo}\label{theodist}
Assume that the measures $\lambda_{1}$ and $\lambda_{2}$ are in
the class $\Reg$, and suppose that $\supp(\lambda_{1})$ and
$\supp(\lambda_{2})$ are regular for the Dirichlet problem. Then
\begin{equation}\label{weakstar1}
\mu_{P_{n}}\stackrel{*}{\longrightarrow}\overline{\mu}_{1}
\in\mathcal{M}_{1}(\Delta_{1}),\qquad\qquad
\mu_{P_{n,2}}\stackrel{*}{\longrightarrow}
\overline{\mu}_{2}\in\mathcal{M}_{1}(\Delta_{2}),
\end{equation}
where
$\overline{\boldsymbol{\mu}}=(\overline{\mu}_{1},\overline{\mu}_{2})$
is the vector equilibrium measure determined by the interaction
matrix
\begin{equation}\label{defninteractionmatrix}
\left[\begin{array}{cc}
1 & -1/4 \\
\\
-1/4 & 1/4\\
\end{array} \right]
\end{equation}
on the system of intervals $(\Delta_{1}, \Delta_{2})$. Therefore,
the limits
\begin{equation}\label{eq:nthrootasympPn}
\lim_{n\rightarrow\infty}|P_{n}(z)|^{1/\lfloor n/3 \rfloor}
=e^{-V^{\overline{\mu}_{1}}(z)},\qquad
z\in\mathbb{C}\setminus\Delta_{1},
\end{equation}
\begin{equation}\label{eq:nthrootasympPn2}
\lim_{n\rightarrow\infty}|P_{n,2}(z)|^{1/\lfloor n/6 \rfloor}
=e^{-V^{\overline{\mu}_{2}}(z)},\qquad
z\in\mathbb{C}\setminus\Delta_{2},
\end{equation}
hold uniformly on compact subsets of the indicated regions.
Moreover,
\begin{equation}\label{asympK1}
\lim_{k\rightarrow\infty}\Big(\int_{0}^{\alpha^{3}}P_{6k+j}^{2}(\tau)\,d\nu_{6k+j}(\tau)\Big)^{1/4k}
=e^{-\omega_{1}^{\overline{\boldsymbol{\mu}}}}, \qquad\mbox{for
all}\quad j=0,\ldots,5,
\end{equation}
\begin{equation}\label{asympK2}
\lim_{k\rightarrow\infty}\Big(\int_{-b^{3}}^{-a^{3}}P_{6k+j,2}^{2}(\tau)\,d\nu_{6k+j,2}(\tau)\Big)^{1/2k}
=e^{-4\omega_{2}^{\overline{\boldsymbol{\mu}}}},\qquad\mbox{for
all}\quad j=0,\ldots,5,
\end{equation}
where
$(\omega_{1}^{\overline{\boldsymbol{\mu}}},\omega_{2}^{\overline{\boldsymbol{\mu}}})$
is the corresponding vector of equilibrium constants, and the
varying measures $d\nu_{6k+j}$ and $d\nu_{6k+j,2}$ are defined in
$(\ref{orthogvaryingmeasures})$ below.
\end{theo}

\begin{co}\label{cortheodist}
Under the same assumptions of Theorem $\ref{theodist}$, let
$\overline{\boldsymbol{\mu}}=(\overline{\mu}_{1},\overline{\mu}_{2})$
be the vector equilibrium measure determined by the interaction
matrix $(\ref{defninteractionmatrix})$ on the system of intervals
$[0,\alpha^3], [-b^3,-a^3]$, and let
$(\omega_{1}^{\overline{\boldsymbol{\mu}}},
\omega_{2}^{\overline{\boldsymbol{\mu}}})$ be the corresponding
vector of equilibrium constants. Consider the probability measures
$\vartheta_{1}\in\mathcal{M}_{1}([0,\alpha])$ and
$\vartheta_{2}\in\mathcal{M}_{1}([-b,-a])$, defined as follows:
\[
\vartheta_{1}(E):=\overline{\mu}_{1}(E^{3}), \quad
E\subset[0,\alpha],\qquad\qquad\qquad
\vartheta_{2}(E):=\overline{\mu}_{2}(E^{3}), \quad
E\subset[-b,-a],
\]
where $E^{3}=\{x^{3}:x\in E\}$. If we denote by $Z_{Q_{n}}$ the
set of all roots of $Q_{n}$ on $(0,\alpha)$, and by $Z_{Q_{n,2}}$
the set of all roots of $Q_{n,2}$ on $(-b,-a)$, then
\[
\frac{1}{n}\sum_{x\in
Z_{Q_{n}}}\delta_{x}\stackrel{*}{\longrightarrow}\frac{1}{3}\,\vartheta_{1},
\qquad\qquad\qquad\frac{1}{n}\sum_{x\in
Z_{Q_{n,2}}}\delta_{x}\stackrel{*}{\longrightarrow}\frac{1}{6}
\,\vartheta_{2}.
\]
The limits
\[
\lim_{n\rightarrow\infty}|Q_{n}(z)|^{1/n}
=e^{-\frac{1}{3}V^{\overline{\mu}_{1}}(z^3)},\qquad
z\in\mathbb{C}\setminus S_{0},
\]
\[
\lim_{n\rightarrow\infty}|Q_{n,2}(z)|^{1/n}
=e^{-\frac{1}{6}V^{\overline{\mu}_{2}}(z^3)},\qquad
z\in\mathbb{C}\setminus S_{1},
\]
hold uniformly on compact subsets of the indicated regions.
Finally, we have
\[
\lim_{k\rightarrow\infty}
\Big(\int_{0}^{\alpha}Q_{3k}^{2}(t)\,\frac{s_{1}(t)}
{Q_{3k,2}(t)}\,dt\Big)^{1/k}=e^{-2\omega_{1}^{\overline{\boldsymbol{\mu}}}}
\]
\[
\lim_{k\rightarrow\infty}
\Big(\int_{0}^{\alpha}Q_{3k+1}^{2}(t)\,\frac{t\,s_{1}(t)}
{Q_{3k+1,2}(t)}\,dt\Big)^{1/k}=e^{-2\omega_{1}^{\overline{\boldsymbol{\mu}}}}
\]
\[
\lim_{k\rightarrow\infty}
\Big(\int_{0}^{\alpha}Q_{3k+2}^{2}(t)\,\frac{s_{1}(t)}
{t\,Q_{3k+2,2}(t)}\,dt\Big)^{1/k}=e^{-2\omega_{1}^{\overline{\boldsymbol{\mu}}}}
\]
\[
\lim_{k\rightarrow\infty}
\Big(\int_{-b}^{-a}Q_{3k,2}^{2}(t)\,\frac{|t\, h_{3k}(t)|}
{|Q_{3k}(t)|}\,s_{2}(t)\,dt\Big)^{1/k}
=e^{-4\omega_{2}^{\overline{\boldsymbol{\mu}}}},
\]
\[
\lim_{k\rightarrow\infty}
\Big(\int_{-b}^{-a}Q_{3k+1,2}^{2}(t)\,\frac{|h_{3k+1}(t)|}
{|Q_{3k+1}(t)|}\,s_{2}(t)\,dt\Big)^{1/k}
=e^{-4\omega_{2}^{\overline{\boldsymbol{\mu}}}},
\]
\[
\lim_{k\rightarrow\infty}
\Big(\int_{-b}^{-a}Q_{3k+2,2}^{2}(t)\,\frac{t^2|h_{3k+2}(t)|}
{|Q_{3k+2}(t)|}\,s_{2}(t)\,dt\Big)^{1/k}
=e^{-4\omega_{2}^{\overline{\boldsymbol{\mu}}}},
\]
where the functions $h_{n}$ are defined in
$(\ref{eq:definitionhn})$ $($see also $(\ref{intreph3k})$$)$.
\end{co}

The following proposition provides a link between the results on
ratio and $n$th root asymptotics.

\begin{prop}\label{relratioandnthrootasymptotics}
Under the same assumptions of Theorem
$\ref{introd:ratioasymptotics}$, the following relations hold:
\begin{equation}\label{relrationthroot1}
V^{\overline{\mu}_{1}}(z)=-\frac{1}{2}\sum_{i=0}^{5}\log
|\widetilde{F}_{1}^{(i)}(z)|,\qquad
z\in\mathbb{C}\setminus[0,\alpha^3],
\end{equation}
\begin{equation}\label{relrationthroot2}
V^{\overline{\mu}_{2}}(z)
=-\sum_{i=0}^{5}\log|\widetilde{F}_{2}^{(i)}(z)|,\qquad
z\in\mathbb{C}\setminus[-b^3,-a^3],
\end{equation}
where $(\overline{\mu}_{1},\overline{\mu}_{2})$ is the vector
equilibrium measure determined by the interaction matrix
$(\ref{defninteractionmatrix})$ on the system of intervals
$[0,\alpha^3], [-b^3,-a^3]$.
\end{prop}

Theorem \ref{theodist} and Proposition
\ref{relratioandnthrootasymptotics} are proved in Section
\ref{sectionnthroot}. Corollary \ref{cortheodist} follows
immediately from Theorem \ref{theodist}, so we omit its proof.

\section{The polynomials $Q_{n}$}\label{sectionQnandzeros}

Let
\begin{equation*}\label{defnS1}
\Sigma_{1}:= \bigcup_{k=0}^{2}\exp(2\pi i k/3)\,(-\infty,0].
\end{equation*}
We may assume that $s_{2}\equiv 0$ on
$(-\infty,0]\setminus[-b,-a]$, and we extend $s_{2}$ to
$\Sigma_{1}$ through the symmetry property
\begin{equation}\label{symrho1}
s_{2}(e^{\frac{2\pi i}{3}}z)=e^{\frac{4\pi i}{3}}s_{2}(z),\quad
z\in \Sigma_{1}.
\end{equation}
Then
\begin{equation}\label{eq:stateprop2}
f(z)=\frac{1}{3}\int_{S_{1}}\frac{s_{2}(t)}{t-z}\,dt=\frac{z^{2}}{3}
\int_{-b^{3}}^{-a^{3}}\frac{s_{2}(\sqrt[3]{\tau})}{(z^{3}-\tau)\,\tau^{2/3}}\,\,d\tau,\qquad
z\in\mathbb{C}\setminus S_{1}.
\end{equation}

\begin{prop}\label{ortPsi1}
The functions $\Psi_{n}$ satisfy the following orthogonality
conditions:
\begin{equation}\label{eq:ortPsi1}
0=\int_{S_{1}}t^{\nu}\,\Psi_{2n}(t)\,s_{2}(t)\,dt,\qquad
\nu=0,\ldots,n-1,
\end{equation}
\begin{equation}\label{eq:ortPsi2}
0=\int_{S_{1}}t^{\nu}\,\Psi_{2n+1}(t)\,s_{2}(t)\,dt,\qquad
\nu=0,\ldots,n-1.
\end{equation}
\end{prop}
\begin{proof}
If $0\leq \nu\leq n-1$, applying the definition of $\Psi_{2n}$ and
Fubini's theorem, we obtain
\[
\int_{S_{1}}t^{\nu}\,\Psi_{2n}(t)\,s_{2}(t)\,dt
=\int_{S_{0}}Q_{2n}(x)\,s_{1}(x)\int_{S_{1}}\frac{t^{\nu}-x^{\nu}+x^{\nu}}{x-t}\,s_{2}(t)\,dt\,dx
=\int_{S_{0}}Q_{2n}(x)\,(p_{\nu}(x)-3 x^{\nu}f(x))\,s_{1}(x)\,dx,
\]
where $p_{\nu}$ is a polynomial of degree at most $n-2$. Using
(\ref{eq:ortQn1})--(\ref{eq:ortQnn}), (\ref{eq:ortPsi1}) follows.
The proof of $(\ref{eq:ortPsi2})$ is identical. \end{proof}

\begin{prop}\label{ortint}
Let $Q_{n}$ be the monic polynomial of smallest degree satisfying
the conditions $(\ref{eq:ortQn1})$--$(\ref{eq:ortQn3})$. If
$d_{n}:=\deg{Q_{n}}$, then
\begin{equation}\label{sympropQn}
Q_{n}(e^{\frac{2\pi i}{3}}z)=e^{\frac{2\pi i
d_{n}}{3}}Q_{n}(z),\qquad\qquad
Q_{n}(z)=\overline{Q_n(\overline{z})}.
\end{equation}
In particular, all the coefficients of $Q_{n}$ are real.
Furthermore, for each $0\leq k\leq n-1$,
\begin{equation}\label{eq:ortintQn1}
0=\int_{0}^{\alpha}t^{k}Q_{2n}(t)(1+e^{2\pi i(k+d_{2n})/3}+e^{4\pi
i(k+d_{2n})/3})\,s_{1}(t)\,dt,
\end{equation}
\begin{equation}\label{eq:ortintQn2}
0=\int_{0}^{\alpha}t^{k}Q_{2n}(t)(1+e^{2\pi
i(k+2+d_{2n})/3}+e^{4\pi i(k+2+d_{2n})/3})\,s_{1}(t) f(t)\,dt,
\end{equation}
\begin{equation}\label{eq:ortintQn4}
0=\int_{0}^{\alpha}t^{k}Q_{2n+1}(t)(1+e^{2\pi
i(k+2+d_{2n+1})/3}+e^{4\pi i(k+2+d_{2n+1})/3})\,s_{1}(t) f(t)\,dt,
\end{equation}
and for each $0\leq k\leq n$,
\begin{equation}\label{eq:ortintQn3}
0=\int_{0}^{\alpha}t^{k}Q_{2n+1}(t)(1+e^{2\pi
i(k+d_{2n+1})/3}+e^{4\pi i(k+d_{2n+1})/3})\,s_{1}(t)\,dt.
\end{equation}
\end{prop}
\begin{proof}
It is easy to check that $Q_{n}(z)$, $Q_{n}(e^{\frac{2\pi
i}{3}}z)/e^{\frac{2\pi i d_{n}}{3}}$ and
$\overline{Q_{n}(\overline{z})}$ satisfy the same orthogonality
conditions. By uniqueness, these polynomials must be equal, so
(\ref{sympropQn}) holds. If we write
(\ref{eq:ortQn1})--(\ref{eq:ortQn3}) in terms of $[0,\alpha]$, we
obtain (\ref{eq:ortintQn1})--(\ref{eq:ortintQn3}).
\end{proof}

\begin{lem}\label{AT}
Let $n_{1}, n_{2}$ be non-negative integers, and assume that
$P_{1}, P_{2}$ are polynomials, not both identically equal to
zero, such that $\deg P_{1}\leq n_{1}-1$ and $\deg P_{2}\leq
n_{2}-1$. Then the functions
\begin{equation}\label{eq:ATsyst1}
H_{1}(t):=P_{1}(t)+P_{2}(t)\sqrt[3]{t}\, f(\sqrt[3]{t}\,),\qquad
t>0,
\end{equation}
\begin{equation}\label{eq:ATsyst2}
H_{2}(t):=P_{1}(t)\,t+P_{2}(t)\sqrt[3]{t}\,
f(\sqrt[3]{t}\,),\qquad t>0,
\end{equation}
have at most $n_{1}+n_{2}-1$ zeros on $(0,\infty)$, counting
multiplicities.
\end{lem}
\begin{proof}
Let $\sigma$ be a finite positive measure with compact support in
$\mathbb{R}$, and let
\[
\widehat{\sigma}(z):=\int\frac{d\sigma(x)}{z-x}.
\]
Lemma 5 in \cite{FidLop} asserts that $\{1, \widehat{\sigma}\}$
forms an AT system on any closed interval
$\Delta\subset\mathbb{R}$ disjoint from $\Co(\supp(\sigma))$, the
convex hull of $\supp(\sigma)$. This means that for any
multi-index $(n_{1},n_{2})\in\mathbb{Z}_{+}^{2}$, and any pair of
polynomials $\pi_{1}, \pi_{2}$ with $\deg \pi_{1}\leq n_{1}-1$,
$\deg \pi_{2}\leq n_{2}-1$, not both identically equal to zero,
the function $\pi_{1}+\pi_{2}\,\widehat{\sigma}$ has at most
$n_{1}+n_{2}-1$ zeros on $\Delta$, counting multiplicities. By
(\ref{eq:stateprop2}) we know that $H_{2}(t)=t\,(P_{1}(t)
+P_{2}(t)\,\widehat{\sigma}(t))$, where $\sigma$ denotes now the
measure $(s_{2}(\sqrt[3]{\tau})/3 \tau^{2/3})\,d\tau$ supported on
$[-b^3,-a^3]$, so the assertion concerning $H_{2}$ is valid.

Let $n_{1}\geq n_{2}$, and suppose that there exist polynomials
$P_{1}, P_{2}$, not both identically equal to zero, such that
$H_{1}$ has at least $n_{1}+n_{2}$ zeros on $(0,\infty)$, counting
multiplicities. We may assume that $P_{2}\not\equiv 0$. Let $T$ be
a polynomial of degree $n_{1}+n_{2}$ that vanishes at
$n_{1}+n_{2}$ zeros of $H_{1}$ on $(0,\infty)$. $H_{1}$ can be
analytically extended onto $\mathbb{C}\setminus[-b^3,-a^3]$,
\[
\frac{H_{1}(z)}{T(z)}=\frac{P_{1}(z)}{T(z)} +\frac{z P_{2}(z)}{3
T(z)}\int_{-b^3}^{-a^3}
\frac{s_{2}(\sqrt[3]{\tau})}{z-\tau}\frac{d\tau}{\tau^{2/3}}
=O\Big(\frac{1}{z^{n_{2}+1}}\Big),\qquad z\rightarrow\infty.
\]
By a standard argument this implies that
\[
0=\int_{-b^3}^{-a^3}\frac{\tau^{\nu+1}
P_{2}(\tau)\,s_{2}(\sqrt[3]{\tau})}{T(\tau)\,\tau^{2/3}}\,d\tau,\qquad
0\leq \nu\leq n_{2}-1,
\]
contradicting the fact that $\deg P_{2}\leq n_{2}-1$. If
$n_{1}<n_{2}$, we use again this argument by contradiction, but
now we divide $H_{1}(z)$ by $T(z)\,\widehat{\sigma}(z)$ instead of
$T(z)$, and use the fact that
$1/\widehat{\sigma}(z)=l(z)+\widehat{\mu}(z)$, where $l(z)$ is a
polynomial of degree one and $\mu$ is a measure of constant sign
supported on $[-b^3,-a^3]$ (see the appendix of \cite{KreinNud}).
\end{proof}

\noindent{\bf Proof of Proposition \ref{propmaximal}.} Assume
first that $n=3l$, $d_{2n}=3j$. Then
$(\ref{eq:ortintQn1})$--$(\ref{eq:ortintQn2})$ reduce to
\[
0=\int_{0}^{\alpha}t^{3k}Q_{2n}(t)\,s_{1}(t)\,dt=\int_{0}^{\alpha}t^{3k}Q_{2n}(t)\,t\,f(t)\,s_{1}(t)\,dt,\qquad
0\leq k\leq l-1.
\]
From (\ref{sympropQn}) and the assumption $d_{2n}=3j$, we deduce
that $Q_{2n}(t)=\widetilde{Q}_{2n}(t^{3})$, for a polynomial
$\widetilde{Q}_{2n}$ of degree $j$. Therefore,
\begin{equation}\label{eq:orttilde2n1}
0=\int_{0}^{\alpha^{3}}\tau^{k}\,\widetilde{Q}_{2n}(\tau)\,
s_{1}(\sqrt[3]{\tau})\,\frac{d\tau}{\tau^{2/3}}
=\int_{0}^{\alpha^{3}}\tau^{k}\,\widetilde{Q}_{2n}(\tau)\,
\sqrt[3]{\tau}\,f(\sqrt[3]{\tau})\,s_{1}(\sqrt[3]{\tau})\,\frac{d\tau}{\tau^{2/3}},\qquad
0\leq k\leq l-1.
\end{equation}

Suppose that $\widetilde{Q}_{2n}$ has $N<2 l$ sign change knots on
$(0,\alpha^{3})$. Let $P_{1}$, $P_{2}$ be polynomials of degree at
most $l-1$, $(P_{1},P_{2})\neq (0,0)$, such that
$H_{1}(t)=P_{1}(t)+P_{2}(t)\sqrt[3]{t}\,f(\sqrt[3]{t})$ has a zero
at each point where $\widetilde{Q}_{2n}$ changes sign on
$(0,\alpha^{3})$, and a zero of order $2 l-1-N$ at $\alpha^{3}$.
By Lemma \ref{AT}, $H_{1}$ has no zeros on $(0,\alpha^{3}]$ other
than the $2 l-1$ prescribed. Combining the two orthogonality
conditions in $(\ref{eq:orttilde2n1})$ we obtain
\[
\int_{0}^{\alpha^{3}}H_{1}(\tau)\,\widetilde{Q}_{2n}(\tau)\,
s_{1}(\sqrt[3]{\tau})\,\frac{d\tau}{\tau^{2/3}}\,d\tau=0.
\]
This contradicts the fact that $H_{1}\,\widetilde{Q}_{2n}$ is
real-valued and has constant sign on $[0,\alpha^{3}]$. Applying
$(\ref{sympropQn})$ we conclude that $Q_{2n}$ has exactly $2 n$
simple zeros on $S_{0}$, $2 n/3$ of them are located on
$(0,\alpha)$, and the remaining zeros are rotations of the zeros
on $(0,\alpha)$ by angles of $2\pi/3$ and $4\pi/3$.

Suppose now that $n=3l$ and $d_{2n}=3j+1$. We will reach a
contradiction. In this case
$Q_{2n}(t)=t\,\widetilde{Q}_{2n}(t^{3})$, for some polynomial
$\widetilde{Q}_{2n}$ of degree $j$. From $(\ref{eq:ortintQn1})$
and $(\ref{eq:ortintQn2})$ we deduce that
\begin{equation}\label{eq:orttilde2n3}
0=\int_{0}^{\alpha^{3}}\tau^{k}\widetilde{Q}_{2n}(\tau)
\,\tau\,s_{1}(\sqrt[3]{\tau})\,\frac{d\tau}{\tau^{2/3}}
=\int_{0}^{\alpha^{3}}\tau^{k}\widetilde{Q}_{2n}(\tau)\,\sqrt[3]{\tau}\,
f(\sqrt[3]{\tau})\,s_{1}(\sqrt[3]{\tau})\,\frac{d\tau}{\tau^{2/3}},\qquad
0\leq k\leq l-1.
\end{equation}
The polynomial $\widetilde{Q}_{2n}$ has $N\leq j$ sign change
knots on $(0,\alpha^3)$. Since $d_{2n}\leq 2n$, we have $j\leq 2
l-1$. Let $P_{1}, P_{2}$ be polynomials of degree at most $l-1$,
not both simultaneously zero, such that
$H_{2}(t)=P_{1}(t)\,t+P_{2}(t)\sqrt[3]{t}\, f(\sqrt[3]{t})$ has a
zero at each point where $\widetilde{Q}_{2n}$ changes sign on
$(0,\alpha^{3})$ and has a zero of order $2 l-1-N$ at
$\alpha^{3}$. The same argument used before but now applied to
$H_{2}$ shows that Lemma \ref{AT} and (\ref{eq:orttilde2n3}) yield
a contradiction. Therefore $d_{2n}=3j+1$ is impossible if $n$ is a
multiple of $3$. Similarly one proves that the assumptions $n=3l$
and $d_{2n}=3j+2$ are not compatible.

The cases $n=3l+1$ and $n=3l+2$ are handled in identical manner,
showing in the first case that $d_{2n}$ is of the form $3j+2$ and
$Q_{2n}$ has $2l$ sign change knots on $(0,\alpha)$, and in the
second case by showing that $d_{2n}$ is of the form $3j+1$ and
$Q_{2n}$ has $2l+1$ sign change knots on $(0,\alpha)$.

The analysis for the polynomials $Q_{2n+1}$ is similar. The
details are left to the reader.\hfill $\Box$

\begin{co}
The polynomials $Q_{n}$ and the functions $\Psi_{n}$ satisfy
\begin{equation}\label{sympropQnrev}
Q_{n}(e^{\frac{2\pi i}{3}}z)=e^{\frac{2\pi i n}{3}}Q_{n}(z),
\end{equation}
\begin{equation}\label{sympropPsin1}
\Psi_{n}(e^{\frac{2\pi i}{3}}z)=e^{-\frac{2\pi
i}{3}(1+2n)}\Psi_{n}(z),
\end{equation}
for all $n\geq 0$.
\end{co}
\begin{proof}
$(\ref{sympropQnrev})$ follows from (\ref{sympropQn}) and
$d_{n}=n$. (\ref{sympropPsin1}) is an immediate consequence of
$(\ref{sympropQnrev})$ and the definition of $\Psi_{n}$.
\end{proof}

\noindent{\bf Proof of Proposition \ref{proprecurrence}.} The
initial conditions (\ref{initialcond}) are immediate to check. For
$n\geq 1$,
\begin{equation}\label{eq:rec1}
zQ_{2n}=Q_{2n+1}+b_{2n}Q_{2n}+b_{2n-1}Q_{2n-1}+b_{2n-2}Q_{2n-2}+\cdots+b_{1}Q_{1}+b_{0}Q_{0},
\end{equation}
and let us show that
\begin{equation}\label{eq:rec2}
b_{2n-3}=b_{2n-4}=\cdots=b_{1}=b_{0}=0,
\end{equation}
and
\begin{equation}\label{eq:rec3}
b_{2n}=b_{2n-1}=0.
\end{equation}

We prove $(\ref{eq:rec2})$ by induction. Let $n\geq 2$. If we
integrate (\ref{eq:rec1}) term by term with respect to
$s_{1}(t)\,dt$, (\ref{eq:ortQn1}) and (\ref{eq:ortQn2}) imply that
$b_{0}=0$. The fact that $b_{1}=0$ follows now by integrating
(\ref{eq:rec1}) term by term with respect to $f(t)\,s_{1}(t)\,dt$.
Assume now that $0=b_{0}=b_{1}=\cdots=b_{2k}=b_{2k+1}=0$ for some
$k\leq n-3$. After multiplying (\ref{eq:rec1}) by $z^{k+1}$ and
integrating the resulting equation first with respect to
$s_{1}(t)\,dt$, and then with respect to $f(t)\,s_{1}(t)\,dt$, we
get $b_{2k+2}=b_{2k+3}=0$ (observe that
$\int_{S_{0}}t^{k+1}Q_{2k+2}(t)\,s_{1}(t)\,dt\neq 0$ and
$\int_{S_{0}}t^{k+1}Q_{2k+3}(t)\,f(t)\,s_{1}(t)\,dt\neq 0$), so
(\ref{eq:rec2}) follows. (\ref{eq:rec3}) is immediate from
(\ref{sympropQnrev}).

Analogously one shows that for $n\geq 1$, $z
Q_{2n+1}=Q_{2n+2}+a_{2n+1}\,Q_{2n-1}$, $a_{2n+1}\in\mathbb{R}$, so
(\ref{recurrence}) is justified. The formulas in (\ref{intrepa2n})
follow directly from (\ref{recurrence}). The positivity of the
recurrence coefficients is proved later in Proposition
\ref{proppositrec}. \hfill $\Box$

\section{The functions of second type $\Psi_{n}$
and associated polynomials $Q_{n,2}$}\label{sectionsecondtypeQn2}

\begin{prop}
The following formula holds:
\begin{equation}\label{eq:Psigeneral}
\Psi_{n}(z)=\int_{0}^{\alpha}\Big(\frac{1}{t-z}+\frac{e^{\frac{2\pi
i n}{3}}}{e^{\frac{2\pi i}{3}}t-z}+\frac{e^{\frac{4\pi i
n}{3}}}{e^{\frac{4\pi i}{3}}t-z}\Big)Q_{n}(t)\,s_{1}(t)\,dt,\qquad
z\notin S_{0}.
\end{equation}
In particular, for any integer $k\geq 0$,
\begin{equation}\label{eq:Psi3k}
\Psi_{3k}(z)=3
z^{2}\int_{0}^{\alpha}\frac{Q_{3k}(t)\,s_{1}(t)}{t^{3}-z^{3}}\,dt
=z^{2}\int_{0}^{\alpha^{3}}\frac{Q_{3k}(\sqrt[3]{\tau})\,s_{1}(\sqrt[3]{\tau})}
{\tau-z^{3}}\,\frac{d\tau}{\tau^{2/3}},
\end{equation}
\begin{equation}\label{eq:Psi3k1}
\Psi_{3k+1}(z)=3\int_{0}^{\alpha}\frac{t^{2}\,Q_{3k+1}(t)\,s_{1}(t)}{t^{3}-z^{3}}\,dt
=\int_{0}^{\alpha^{3}}\frac{Q_{3k+1}(\sqrt[3]{\tau})\,s_{1}(\sqrt[3]{\tau})}
{\tau-z^{3}}\,d\tau,
\end{equation}
\begin{equation}\label{eq:Psi3k2}
\Psi_{3k+2}(z)=3
z\int_{0}^{\alpha}\frac{t\,Q_{3k+2}(t)\,s_{1}(t)}{t^{3}-z^{3}}\,dt
=z\int_{0}^{\alpha^{3}}\frac{Q_{3k+2}(\sqrt[3]{\tau})\,s_{1}(\sqrt[3]{\tau})}
{\tau-z^{3}}\,\frac{d\tau}{\tau^{1/3}}.
\end{equation}
\end{prop}
\begin{proof}
The definition of $\Psi_{n}$ and the symmetry properties
(\ref{eq:symsigma1}) and (\ref{sympropQnrev}) give directly
(\ref{eq:Psigeneral}).
\end{proof}
\begin{prop}\label{proportQdetailed}
For any integer $l\geq 0$, the following orthogonality conditions
hold:
\begin{equation}\label{eq:ortQ6l}
0=\int_{0}^{\alpha^{3}}\tau^{k}\,Q_{6l}(\sqrt[3]{\tau})
\,s_{1}(\sqrt[3]{\tau})\,\frac{d\tau}{\tau^{2/3}},\qquad 0\leq
k\leq l-1,
\end{equation}
\begin{equation}\label{eq:ortQ6l1}
0=\int_{0}^{\alpha^{3}}\tau^{k}\,Q_{6l+1}(\sqrt[3]{\tau})
\,s_{1}(\sqrt[3]{\tau})\,d\tau,\qquad 0\leq k\leq l-1,
\end{equation}
\begin{equation}\label{eq:ortQ6l2}
0=\int_{0}^{\alpha^{3}}\tau^{k}\,Q_{6l+2}(\sqrt[3]{\tau})
\,s_{1}(\sqrt[3]{\tau})\,\frac{d\tau}{\tau^{1/3}},\qquad 0\leq
k\leq l-1,
\end{equation}
\begin{equation}\label{eq:ortQ6l3}
0=\int_{0}^{\alpha^{3}}\tau^{k}\,Q_{6l+3}(\sqrt[3]{\tau})
\,s_{1}(\sqrt[3]{\tau})\,\frac{d\tau}{\tau^{2/3}},\qquad 0\leq
k\leq l,
\end{equation}
\begin{equation}\label{eq:ortQ6l4}
0=\int_{0}^{\alpha^{3}}\tau^{k}\,Q_{6l+4}(\sqrt[3]{\tau})
\,s_{1}(\sqrt[3]{\tau})\,d\tau,\qquad 0\leq k\leq l-1,
\end{equation}
\begin{equation}\label{eq:ortQ6l5}
0=\int_{0}^{\alpha^{3}}\tau^{k}\,Q_{6l+5}(\sqrt[3]{\tau})
\,s_{1}(\sqrt[3]{\tau})\,\frac{d\tau}{\tau^{1/3}},\qquad 0\leq
k\leq l.
\end{equation}
\end{prop}
\begin{proof}
Follows immediately from Proposition \ref{ortint} and the fact
that $d_{n}=n$ for all $n\geq 0$.
\end{proof}
\begin{co}
The following estimates are valid as $z\rightarrow\infty$:
\begin{equation}\label{eq:ordPsi6linf}
\Psi_{6l}(z)=O\Big(\frac{1}{z^{3l+1}}\Big),\qquad
\Psi_{6l+2}(z)=O\Big(\frac{1}{z^{3l+2}}\Big),\qquad
\Psi_{6l+4}(z)=O\Big(\frac{1}{z^{3l+3}}\Big),
\end{equation}
\begin{equation}\label{eq:ordPsi6l1inf}
\Psi_{6l+1}(z)=O\Big(\frac{1}{z^{3l+3}}\Big),\qquad
\Psi_{6l+3}(z)=O\Big(\frac{1}{z^{3l+4}}\Big),\qquad
\Psi_{6l+5}(z)=O\Big(\frac{1}{z^{3l+5}}\Big).
\end{equation}
\end{co}
\begin{proof}
By (\ref{eq:BVPPsin1}) we know that $\Psi_{2n}(z)=O(1/z^{n+1})$,
which implies (\ref{eq:ordPsi6linf}). We can improve the estimate
$\Psi_{2n+1}(z)=O(1/z^{n+2})$ given in (\ref{eq:BVPPsin1}). From
Proposition \ref{proportQdetailed} we deduce:
\[
\int\frac{Q_{6l+1}(\sqrt[3]{\tau})
\,s_{1}(\sqrt[3]{\tau})}{\tau-z}\,d\tau=O\Big(\frac{1}{z^{l+1}}\Big),\quad\,\,\,
\int\frac{Q_{6l+3}(\sqrt[3]{\tau})
\,s_{1}(\sqrt[3]{\tau})}{\tau-z}\,\frac{d\tau}{\tau^{2/3}}=O\Big(\frac{1}{z^{l+2}}\Big),\quad\,\,\,
\int\frac{Q_{6l+5}(\sqrt[3]{\tau})
\,s_{1}(\sqrt[3]{\tau})}{\tau-z}\,\frac{d\tau}{\tau^{1/3}}=O\Big(\frac{1}{z^{l+2}}\Big).
\]
We conclude the proof now by taking into account the
representations (\ref{eq:Psi3k})--(\ref{eq:Psi3k2}).
\end{proof}

It is convenient to rewrite the orthogonality conditions in
Proposition \ref{ortPsi1} in terms of the interval $(-b^3,-a^3)$.
Applying the symmetry properties (\ref{sympropPsin1}) and
(\ref{symrho1}), we obtain:

\begin{prop}\label{ortPsinorm}
The functions $\Psi_{n}$ satisfy:
\[
0=\int_{-b}^{-a}t^{\nu}\,\Psi_{2n}(t)\,(1+e^{\frac{2\pi
i}{3}\,(\nu-4n-1)}+e^{\frac{4\pi
i}{3}\,(\nu-4n-1)})\,s_{2}(t)\,dt,\qquad \nu=0,\ldots,n-1,
\]
\[
0=\int_{-b}^{-a}t^{\nu}\,\Psi_{2n+1}(t)\,(1+e^{\frac{2\pi
i}{3}\,(\nu-n)}+e^{\frac{4\pi
i}{3}\,(\nu-n)})\,s_{2}(t)\,dt,\qquad \nu=0,\ldots,n-1.
\]
In particular, for any integer $l\geq 0$,
\begin{equation}\label{eq:ortintPsi6l}
0=\int_{-b^3}^{-a^3}\tau^{k}\,\Psi_{6l+j}(\sqrt[3]{\tau})
\,s_{2}(\sqrt[3]{\tau})\,\frac{d\tau}{\tau^{1/3}},\qquad 0\leq
k\leq l-1,\quad j=0,3,
\end{equation}
\begin{equation}\label{eq:ortintPsi6l2}
0=\int_{-b^3}^{-a^3}\tau^{k}\,\Psi_{6l+2+j}(\sqrt[3]{\tau})
\,s_{2}(\sqrt[3]{\tau})\,d\tau,\qquad 0\leq k\leq l-1,\quad j=0,3,
\end{equation}
\begin{equation}\label{eq:ortintPsi6l1}
0=\int_{-b^3}^{-a^3}\tau^{k}\,\Psi_{6l+1}(\sqrt[3]{\tau})
\,s_{2}(\sqrt[3]{\tau})\,\frac{d\tau}{\tau^{2/3}},\qquad 0\leq
k\leq l-1,
\end{equation}
\begin{equation}\label{eq:ortintPsi6l4}
0=\int_{-b^3}^{-a^3}\tau^{k}\,\Psi_{6l+4}(\sqrt[3]{\tau})
\,s_{2}(\sqrt[3]{\tau})\,\frac{d\tau}{\tau^{2/3}},\qquad 0\leq
k\leq l,
\end{equation}
\end{prop}

As a consequence of
(\ref{eq:ortintPsi6l})--(\ref{eq:ortintPsi6l4}) we have:
\begin{co}\label{numerocambiosdesignoPsi}
For each $j\in\{0,1,2,3,5\}$, the function $\Psi_{6l+j}$ has at
least $l$ sign change knots in the interval $(-b,-a)$, and the
function $\Psi_{6l+4}$ has at least $l+1$ sign change knots in the
interval $(-b,-a)$. Therefore the functions $\Psi_{6l+j},
j\in\{0,1,2,3,5\}$ have at least $3l$ zeros, counting
multiplicities, in $\mathbb{C}\setminus S_{0}$, and $\Psi_{6l+4}$
has at least $3l+3$ zeros, counting multiplicities, in
$\mathbb{C}\setminus S_{0}$.
\end{co}

Observe that the function $\Psi_{n}$ satisfies the property
$\Psi_{n}(\overline{z})=-\overline{\Psi_{n}(z)}$,
$z\in\mathbb{C}\setminus S_{0}$. Considering this fact, Corollary
\ref{numerocambiosdesignoPsi}, and the symmetry property
(\ref{sympropPsin1}), it is easy to see that if $\Psi_{6l+j},
j\in\{0,1,2,3,5\},$ has more than $3l$ zeros in
$\mathbb{C}\setminus S_{0}$, counting multiplicities, then we can
find a polynomial $R_{6l+j}$ with real coefficients and degree at
least $3l+3$ satisfying:
\begin{equation}\label{eq:symR6lj}
R_{6l+j}(z)=R_{6l+j}(e^{\frac{2\pi i}{3}}z),\quad
z\in\mathbb{C},\qquad\mbox{and}\,\,\qquad\frac{\Psi_{6l+j}}{R_{6l+j}}\in
H(\overline{\mathbb{C}}\setminus S_{0}).
\end{equation}
Similarly, if we assume that $\Psi_{6l+4}$ has more than $3l+3$
zeros in $\mathbb{C}\setminus S_{0}$, counting multiplicities,
then there exists a polynomial $R_{6l+4}$ with real coefficients
and degree at least $3l+6$ such that (\ref{eq:symR6lj}) hold for
$j=4$.

\vspace{0.2cm}

\noindent{\bf Proof of Proposition \ref{propzerosPsi1}.} Suppose
that $\Psi_{6l}$ has more than $3l$ zeros in $\mathbb{C}\setminus
S_{0}$, counting multiplicities. Let $R_{6l}$ be a polynomial with
real coefficients and degree at least $3l+3$ satisfying
(\ref{eq:symR6lj}). By (\ref{eq:ordPsi6linf}),
$\Psi_{6l}(z)/R_{6l}(z)=O(1/z^{6l+4})$ as $z\rightarrow\infty$.

Let $\Gamma$ be a Jordan curve surrounding $S_{0}$ such that the
zeros of $R_{6l}$ lie outside $\Gamma$. By Cauchy's theorem,
Fubini's theorem, and Cauchy's integral formula, for
$\nu=0,\ldots,6l+2,$
\[
0=\int_{\Gamma}z^{\nu}\,\frac{\Psi_{6l}(z)}{R_{6l}(z)}\,dz
=\int_{\Gamma}\frac{z^{\nu}}{R_{6l}(z)}\frac{1}{2\pi
i}\int_{0}^{\alpha}\Big(\frac{1}{t-z}+\frac{1}{e^{\frac{2\pi
i}{3}}t-z}+\frac{1}{e^{\frac{4\pi
i}{3}}t-z}\Big)Q_{6l}(t)\,s_{1}(t)\,dt\,dz
\]
\[
=\int_{0}^{\alpha}t^{\nu}\Big[\frac{1}{R_{6l}(t)}+\frac{e^{2\pi
i\,\nu/3}}{R_{6l}(e^{\frac{2\pi i}{3}}t)}+\frac{e^{4\pi
i\,\nu/3}}{R_{6l}(e^{\frac{4\pi
i}{3}}t)}\Big]\,Q_{6l}(t)\,s_{1}(t)\,dt,
\]
and applying (\ref{eq:symR6lj}), we obtain
\[
0=\int_{0}^{\alpha}t^{3k}\,Q_{6l}(t)\,\frac{s_{1}(t)}{R_{6l}(t)}\,dt,\qquad
0\leq k\leq 2l.
\]
Consequently, $Q_{6l}$ has at least $2l+1$ sign change knots in
$(0,\alpha)$, contradicting Proposition \ref{propmaximal}. This
and Corollary \ref{numerocambiosdesignoPsi} prove the claim for
$n=6l$. In the remaining cases we use the same argument. Indeed,
if $\Psi_{6l+j}, j\in\{1,2,3,5\}$, has more than $3l$ zeros in
$\mathbb{C}\setminus S_{0}$ and $\Psi_{6l+4}$ has more than $3l+3$
zeros in $\mathbb{C}\setminus S_{0}$, counting multiplicities,
then we know (see discussion after Corollary
\ref{numerocambiosdesignoPsi}) that we can select polynomials
$R_{6l+j}, 1\leq j\leq 5$ satisfying (\ref{eq:symR6lj}) such that,
as $z\rightarrow\infty$:
\[
\frac{\Psi_{6l+1}(z)}{R_{6l+1}(z)}=O\Big(\frac{1}{z^{6l+6}}\Big),\quad
\frac{\Psi_{6l+2}(z)}{R_{6l+2}(z)}=O\Big(\frac{1}{z^{6l+5}}\Big),\quad
\frac{\Psi_{6l+3}(z)}{R_{6l+3}(z)}=O\Big(\frac{1}{z^{6l+7}}\Big),\quad
\frac{\Psi_{6l+4}(z)}{R_{6l+4}(z)}=O\Big(\frac{1}{z^{6l+9}}\Big),\quad
\frac{\Psi_{6l+5}(z)}{R_{6l+5}(z)}=O\Big(\frac{1}{z^{6l+8}}\Big).
\]
These estimates lead to the orthogonality conditions
\[
0=\int_{0}^{\alpha}t^{3k+2}\,Q_{6l+1}(t)\,\frac{s_{1}(t)}{R_{6l+1}(t)}\,dt
=\int_{0}^{\alpha}t^{3k+1}\,Q_{6l+2}(t)\,\frac{s_{1}(t)}{R_{6l+2}(t)}\,dt,\qquad\,\,
0\leq k\leq 2l,
\]
\[
0=\int_{0}^{\alpha}t^{3k}\,Q_{6l+3}(t)\,\frac{s_{1}(t)}{R_{6l+3}(t)}\,dt
=\int_{0}^{\alpha}t^{3k+2}\,Q_{6l+4}(t)\,\frac{s_{1}(t)}{R_{6l+4}(t)}\,dt
=\int_{0}^{\alpha}t^{3k+1}\,Q_{6l+5}(t)\,\frac{s_{1}(t)}{R_{6l+5}(t)}\,dt,\qquad\,\,
0\leq k\leq 2l+1,
\]
which contradict the number of zeros that the polynomials
$Q_{6l+j}, 1\leq j\leq 5,$ have on $(0,\alpha)$ (see Proposition
\ref{propmaximal}). \hfill $\Box$

\vspace{0.2cm}

Recall that $Q_{n,2}$ is defined as the monic polynomial whose
zeros coincide with the finite zeros of $\Psi_{n}$ outside
$S_{0}$. The argument shown above proves the following:

\begin{prop}\label{prop2}
For each $j\in\{0,1,2,3,5\}$, $\deg(Q_{6l+j,2})=3l$, and
$\deg(Q_{6l+4,2})=3l+3$. Furthermore,
\begin{equation}\label{eq:ortQ6lvar}
0=\int_{0}^{\alpha}t^{3k}\,Q_{3l}(t)\,\frac{s_{1}(t)}{Q_{3l,2}(t)}\,dt,\qquad
0\leq k\leq l-1,
\end{equation}
\begin{equation}\label{eq:ortQ6l1var}
0=\int_{0}^{\alpha}t^{3k+2}\,Q_{3l+1}(t)\,\frac{s_{1}(t)}{Q_{3l+1,2}(t)}\,dt,\qquad
0\leq k\leq l-1,
\end{equation}
\begin{equation}\label{eq:ortQ6l2var}
0=\int_{0}^{\alpha}t^{3k+1}\,Q_{3l+2}(t)\,\frac{s_{1}(t)}{Q_{3l+2,2}(t)}\,dt,\qquad
0\leq k\leq l-1,
\end{equation}
\end{prop}

\begin{prop}
The following formulas are valid for $z\in\mathbb{C}\setminus
S_{0}$. If $q$ is a polynomial of degree at most $3k$, then
\begin{equation}\label{eq:rep1}
\frac{q(z)\Psi_{3k}(z)}{Q_{3k,2}(z)}=\int_{0}^{\alpha}\frac{Q_{3k}(x)\,s_{1}(x)}
{Q_{3k,2}(x)}\Big(\frac{q(x)}{x-z}+\frac{q(e^{\frac{2\pi
i}{3}}x)}{e^{\frac{2\pi i}{3}}x-z}+\frac{q(e^{\frac{4\pi
i}{3}}x)}{e^{\frac{4\pi i}{3}}x-z}\Big)dx.
\end{equation}
If $\deg(q)\leq 3k+2$, then
\begin{equation}\label{eq:rep2}
\frac{q(z)\Psi_{3k+1}(z)}{Q_{3k+1,2}(z)}
=\int_{0}^{\alpha}\frac{Q_{3k+1}(x)\,s_{1}(x)}
{Q_{3k+1,2}(x)}\Big(\frac{q(x)}{x-z}+\frac{e^{\frac{2\pi
i}{3}}q(e^{\frac{2\pi i}{3}}x)}{e^{\frac{2\pi
i}{3}}x-z}+\frac{e^{\frac{4\pi i}{3}}q(e^{\frac{4\pi
i}{3}}x)}{e^{\frac{4\pi i}{3}}x-z}\Big)dx.
\end{equation}
If $\deg(q)\leq 3k+1$, then
\begin{equation}\label{eq:rep3}
\frac{q(z)\Psi_{3k+2}(z)}{Q_{3k+2,2}(z)}
=\int_{0}^{\alpha}\frac{Q_{3k+2}(x)\,s_{1}(x)}{Q_{3k+2,2}(x)}
\Big(\frac{q(x)}{x-z}+\frac{e^{\frac{4\pi i}{3}}q(e^{\frac{2\pi
i}{3}}x)}{e^{\frac{2\pi i}{3}}x-z}+\frac{e^{\frac{2\pi
i}{3}}q(e^{\frac{4\pi i}{3}}x)}{e^{\frac{4\pi i}{3}}x-z}\Big)dx.
\end{equation}
In particular, we have
\begin{equation}\label{eq:rep4}
\frac{Q_{3k}(z)\Psi_{3k}(z)}{Q_{3k,2}(z)}=3
z^{2}\int_{0}^{\alpha}\frac{Q_{3k}^{2}(x)}{Q_{3k,2}(x)}\,\frac{s_{1}(x)}{x^{3}-z^{3}}\,dx,
\end{equation}
\begin{equation}\label{eq:rep5}
\frac{Q_{3k+1}(z)\Psi_{3k+1}(z)}{Q_{3k+1,2}(z)}=3
z\int_{0}^{\alpha}\frac{Q_{3k+1}^{2}(x)}{Q_{3k+1,2}(x)}\,\frac{x\,s_{1}(x)}{x^{3}-z^{3}}\,dx,
\end{equation}
\begin{equation}\label{eq:rep6}
\frac{Q_{3k+2}(z)\Psi_{3k+2}(z)}{Q_{3k+2,2}(z)}=3
z^{3}\int_{0}^{\alpha}\frac{Q_{3k+2}^{2}(x)}{Q_{3k+2,2}(x)}\,\frac{s_{1}(x)}{x(x^{3}-z^{3})}\,dx.
\end{equation}
\end{prop}
\begin{proof}
By (\ref{eq:ordPsi6linf})--(\ref{eq:ordPsi6l1inf}) and Proposition
\ref{prop2}, we know that if $q$ is a polynomial of degree at most
$3k$, then
\begin{equation}\label{eq:aux11}
\frac{q(z)\Psi_{3k}(z)}{Q_{3k,2}(z)}=O\Big(\frac{1}{z}\Big),\qquad
z\rightarrow\infty.
\end{equation}
For $z\in\mathbb{C}\setminus S_{0}$, let $\Gamma$ be a Jordan
curve surrounding $S_{0}$ and oriented clockwise, so that $z$ and
the zeros of $Q_{3k,2}$ lie outside $\Gamma$. From
(\ref{eq:aux11}) and (\ref{eq:Psigeneral}) it follows that
\[
\frac{q(z)\Psi_{3k}(z)}{Q_{3k,2}(z)}=\frac{1}{2\pi
i}\int_{\Gamma}\frac{q(t)\Psi_{3k}(t)}{Q_{3k,2}(t)}\frac{dt}{t-z}
=\int_{0}^{\alpha}Q_{3k}(x)\,s_{1}(x)\frac{1}{2\pi
i}\int_{\Gamma}\frac{q(t)}{Q_{3k,2}(t)(t-z)}\Big[\frac{1}{x-t}+\frac{1}{e^{\frac{2\pi
i}{3}}x-t}+\frac{1}{e^{\frac{4\pi i}{3}}x-t}\Big]dt\,dx
\]
\[
=\int_{0}^{\alpha}\frac{Q_{3k}(x)\,s_{1}(x)}{Q_{3k,2}(x)}\Big(\frac{q(x)}{x-z}+\frac{q(e^{\frac{2\pi
i}{3}}x)}{e^{\frac{2\pi i}{3}}x-z}+\frac{q(e^{\frac{4\pi
i}{3}}x)}{e^{\frac{4\pi i}{3}}x-z}\Big)dx,
\]
where in the last equality we used that
$Q_{3k,2}(t)=Q_{3k,2}(e^{\frac{2\pi
i}{3}}t)=Q_{3k,2}(e^{\frac{4\pi i}{3}}t)$. This proves
(\ref{eq:rep1}). The proofs of (\ref{eq:rep2})--(\ref{eq:rep3})
are identical. To obtain (\ref{eq:rep4}) and (\ref{eq:rep5}), we
replace $q$ in formulas (\ref{eq:rep1}) and (\ref{eq:rep2}) by
$Q_{3k}$ and $Q_{3k+1}$, respectively. Formula (\ref{eq:rep6})
follows from (\ref{eq:rep3}) by taking $q(z)=Q_{3k+2}(z)/z$.
\end{proof}

\begin{prop}\label{proppositrec}
The recurrence coefficients $\{a_{n}\}_{n\geq 2}^{\infty}$ that
appear in $(\ref{recurrence})$ are all positive.
\end{prop}
\begin{proof}
To prove that $a_{2n}$ is positive it suffices to show that
$\int_{0}^{\alpha}t^{n}\,Q_{2n}(t)\,s_{1}(t)\,dt>0$ for all $n\geq
0$. Let $n=3l$. Since $\deg(t^{3l}\,Q_{6l,2})=6l$, by
(\ref{eq:ortQ6lvar}) we obtain
\[
\int_{0}^{\alpha}t^{3l}\,Q_{6l}(t)\,s_{1}(t)\,dt
=\int_{0}^{\alpha}t^{3l}\,Q_{6l}(t)\,Q_{6l,2}(t)\,\frac{s_{1}(t)}{Q_{6l,2}(t)}\,dt
=\int_{0}^{\alpha}Q_{6l}^{2}(t)\,\frac{s_{1}(t)}{Q_{6l,2}(t)}\,dt>0.
\]
For $n=3l+1$, using (\ref{eq:ortQ6l2var}) and
$\deg{(t^{3l+2}Q_{6l+2,2})}=6l+2$, we get
\[
\int_{0}^{\alpha}t^{3l+1}\,Q_{6l+2}(t)\,s_{1}(t)\,dt
=\int_{0}^{\alpha}t^{3l+2}\,Q_{6l+2,2}(t)\,Q_{6l+2}(t)\,
\frac{s_{1}(t)}{t\,Q_{6l+2,2}(t)}\,dt=\int_{0}^{\alpha}Q_{6l+2}^{2}(t)\,
\frac{s_{1}(t)}{t\,Q_{6l+2,2}(t)}\,dt>0.
\]
Finally, for $n=3l+2$, applying (\ref{eq:ortQ6l1var}) and
$\deg{(t^{3l+1}Q_{6l+4,2})}=6l+4$, we obtain
\[
\int_{0}^{\alpha}t^{3l+2}\,Q_{6l+4}(t)\,s_{1}(t)\,dt=\int_{0}^{\alpha}t^{3l+1}\,Q_{6l+4,2}(t)\,Q_{6l+4}(t)
\frac{t\,s_{1}(t)}{Q_{6l+4,2}(t)}\,dt=\int_{0}^{\alpha}Q_{6l+4}^{2}(t)
\frac{t\,s_{1}(t)}{Q_{6l+4,2}(t)}\,dt>0.
\]

It is easy to see that the functions $\Psi_{n}$ satisfy the same
recurrence relation (\ref{recurrence}). In particular,
\begin{equation}\label{eq:posita6l1}
t\Psi_{2n+1}(t)=\Psi_{2n+2}(t)+a_{2n+1}\Psi_{2n-1}(t).
\end{equation}
Using Proposition \ref{ortPsinorm}, if we multiply
(\ref{eq:posita6l1}) by an appropriate power of $t$ and integrate,
we obtain
\[
\int_{-b}^{-a}t^{3l}\,\Psi_{6l+1}(t)\,s_{2}(t)\,dt=a_{6l+1}
\,\int_{-b}^{-a}t^{3l-1}\,\Psi_{6l-1}(t)\,s_{2}(t)\,dt,\quad
\int_{-b}^{-a}t^{3l+1}\,\Psi_{6l+3}(t)\,s_{2}(t)\,dt=a_{6l+3}
\,\int_{-b}^{-a}t^{3l}\,\Psi_{6l+1}(t)\,s_{2}(t)\,dt,
\]
\[
\int_{-b}^{-a}t^{3l+2}\,\Psi_{6l+5}(t)\,s_{2}(t)\,dt=a_{6l+5}
\,\int_{-b}^{-a}t^{3l+1}\,\Psi_{6l+3}(t)\,s_{2}(t)\,dt.
\]
On the other hand, it is easy to deduce from
(\ref{eq:rep4})--(\ref{eq:rep6}) that if $t<0$, then
\begin{equation}\label{eq:sign1}
\sign\Big(\frac{\Psi_{3k}(t)}{Q_{3k,2}(t)}\Big)=(-1)^{3k},
\qquad\sign\Big(\frac{\Psi_{3k+1}(t)}{Q_{3k+1,2}(t)}\Big)=(-1)^{3k},
\qquad\sign\Big(\frac{\Psi_{3k+2}(t)}{Q_{3k+2,2}(t)}\Big)=(-1)^{3k+1}.
\end{equation}
Observe that since $\deg{Q_{6l-1,2}}=3l-3$ and
$\deg{Q_{6l+1,2}}=\deg{Q_{6l+3,2}}=3l$, by the orthogonality
conditions satisfied by the functions $\Psi_{2n+1}$ and
(\ref{eq:sign1}), we obtain:
\[
\int_{-b}^{-a}t^{3l-1}\,\Psi_{6l-1}(t)\,s_{2}(t)\,dt=\int_{-b}^{-a}Q_{6l-1,2}(t)\Psi_{6l-1}(t)
\,t^{2}\,s_{2}(t)\,dt
=\int_{-b}^{-a}Q_{6l-1,2}^{2}(t)\frac{\Psi_{6l-1}(t)}{Q_{6l-1,2}(t)}
\,t^{2}\,s_{2}(t)\,dt>0,
\]
\[
\int_{-b}^{-a}t^{3l}\,\Psi_{6l+1}(t)\,s_{2}(t)\,dt
=\int_{-b}^{-a}Q_{6l+1,2}(t)\Psi_{6l+1}(t)\,s_{2}(t)\,dt
=\int_{-b}^{-a}Q_{6l+1,2}^{2}(t)\frac{\Psi_{6l+1}(t)}{Q_{6l+1,2}(t)}\,s_{2}(t)\,dt>0,
\]
\[
\int_{-b}^{-a}t^{3l+1}\Psi_{6l+3}(t)\,s_{2}(t)\,dt=
\int_{-b}^{-a}Q_{6l+3,2}(t)\Psi_{6l+3}(t)\,t\,s_{2}(t)\,dt
=\int_{-b}^{-a}Q_{6l+3,2}^{2}(t)\frac{\Psi_{6l+3}(t)}{Q_{6l+3,2}(t)}\,t\,s_{2}(t)\,dt>0.
\]
This shows that $a_{2n+1}>0$ for all $n\geq 1$.\end{proof}

\section{Interlacing properties of the zeros of $Q_{n}$
and $\Psi_{n}$}\label{sectioninterlacing}

\begin{prop}\label{propsimplezeroscomb}
Let $A,B\in \mathbb{R}$ be two constants such that $|A|+|B|>0$,
and let
\begin{equation}\label{defnHn}
Y_{n}(z):=A\, z \Psi_{n}(z)+B\, \Psi_{n+1}(z),
\end{equation}
\begin{equation}\label{defnTn}
T_{n}(z):=A\, z Q_{n}(z)+B\, Q_{n+1}(z).
\end{equation}
Then, for every $n\geq 0$, the function $Y_{n}$ has only simple
zeros on $(-\infty,0)$. Similarly, for every $n\geq 0$, the
polynomial $T_{n}$ has only simple zeros on $(0,\alpha)$.
\end{prop}
\begin{proof}
From $(\ref{eq:ortintPsi6l})$--$(\ref{eq:ortintPsi6l4})$ it
follows that
\[
0=\int_{-b^{3}}^{-a^{3}}\tau^{k}\,Y_{6l+1}(\sqrt[3]{\tau})
\,s_{2}(\sqrt[3]{\tau})\,d\tau,\quad 0\leq k\leq l-2,\qquad\,\,\,
0=\int_{-b^{3}}^{-a^{3}}\tau^{k}\,Y_{6l+4}(\sqrt[3]{\tau})
\,s_{2}(\sqrt[3]{\tau})\,d\tau,\quad 0\leq k\leq l-1,
\]
\[
0=\int_{-b^{3}}^{-a^{3}}\tau^{k}\,Y_{6l+j}(\sqrt[3]{\tau})
\,s_{2}(\sqrt[3]{\tau})\,\frac{d\tau}{\tau^{2/3}},\quad 0\leq
k\leq l-1,\quad j=0,3,
\]
\[
0=\int_{-b^{3}}^{-a^{3}}\tau^{k}\,Y_{6l+2+j}(\sqrt[3]{\tau})
\,s_{2}(\sqrt[3]{\tau})\,\frac{d\tau}{\tau^{1/3}},\quad 0\leq
k\leq l-1,\quad j=0,3.
\]

Consequently, for each $j\in\{0,2,3,4,5\}$, the function
$Y_{6l+j}$ has at least $l$ sign change knots in $(-b,-a)$, and
$Y_{6l+1}$ has at least $l-1$ sign change knots in $(-b,-a)$. From
(\ref{sympropPsin1}) it follows that for every $n$,
$Y_{n}(e^{\frac{2\pi i}{3}}z)=C_{n}\,Y_{n}(z)$, where $C_{n}$
denotes a constant. Therefore, the functions $Y_{6l+j},
j\in\{0,2,3,4,5\}$ have at least $3l$ zeros on $S_{1}$, and
$Y_{6l+1}$ has at least $3l-3$ zeros on $S_{1}$. For each $0\leq
j\leq 5$, let $R_{6l+j}$ denote the monic polynomial whose zeros
coincide with the zeros of $Y_{6l+j}$ on
$\Sigma_{1}\setminus\{0\}$. Then $R_{6l+j}$ satisfies
(\ref{eq:symR6lj}), $Y_{6l+j}/R_{6l+j}\in
H(\overline{\mathbb{C}}\setminus S_{0})$, and using
(\ref{eq:ordPsi6linf})--(\ref{eq:ordPsi6l1inf}) we deduce that as
$z\rightarrow\infty$:
\[
\frac{Y_{6l}(z)}{R_{6l}(z)}=O\Big(\frac{1}{z^{6l}}\Big),\qquad
\frac{Y_{6l+1}(z)}{R_{6l+1}(z)}=O\Big(\frac{1}{z^{6l-1}}\Big),\qquad
\frac{Y_{6l+2}(z)}{R_{6l+2}(z)}=O\Big(\frac{1}{z^{6l+1}}\Big),
\]
\[
\frac{Y_{6l+3}(z)}{R_{6l+3}(z)}=O\Big(\frac{1}{z^{6l+3}}\Big),\qquad
\frac{Y_{6l+4}(z)}{R_{6l+4}(z)}=O\Big(\frac{1}{z^{6l+2}}\Big),\qquad
\frac{Y_{6l+5}(z)}{R_{6l+5}(z)}=O\Big(\frac{1}{z^{6l+4}}\Big).
\]
Let $\Gamma$ again denote a Jordan curve surrounding $S_{0}$, such
that the zeros of the polynomials $R_{6l+j}$ lie outside $\Gamma$.
By (\ref{eq:Psigeneral}),
\[
0=\int_{\Gamma} z^{\nu}\,\frac{Y_{6l}(z)}{R_{6l}(z)}\,dz=
\int_{0}^{\alpha}x^{\nu}\,T_{6l}(x)\,(1+e^{2\pi
i(\nu+1)/3}+e^{4\pi
i(\nu+1)/3})\frac{s_{1}(x)}{R_{6l}(x)}\,dx,\quad\nu=0,\ldots,6l-2,
\]
which is equivalent to
\begin{equation}\label{ortT6l}
0=\int_{0}^{\alpha}x^{3k+2}\,T_{6l}(x)\frac{s_{1}(x)}{R_{6l}(x)}\,dx,\quad
0\leq k\leq 2l-2.
\end{equation}
Similarly we obtain:
\begin{equation}\label{ortT6l1}
0=\int_{0}^{\alpha}x^{3k+1}\,T_{6l+1}(x)\frac{s_{1}(x)}{R_{6l+1}(x)}\,dx,\quad
0\leq k\leq 2l-2,\qquad
0=\int_{0}^{\alpha}x^{3k}\,T_{6l+5}(x)\frac{s_{1}(x)}{R_{6l+5}(x)}\,dx,\quad
0\leq k\leq 2l.
\end{equation}
\begin{equation}\label{ortT6l2}
0=\int_{0}^{\alpha}x^{3k}\,T_{6l+2}(x)\frac{s_{1}(x)}{R_{6l+2}(x)}\,dx
=\int_{0}^{\alpha}x^{3k+2}\,T_{6l+3}(x)\frac{s_{1}(x)}{R_{6l+3}(x)}\,dx
=\int_{0}^{\alpha}x^{3k+1}\,T_{6l+4}(x)\frac{s_{1}(x)}{R_{6l+4}(x)}\,dx,\qquad
0\leq k\leq 2l-1,
\end{equation}
From (\ref{ortT6l}) it follows that $T_{6l}$ has at least $2l-1$
sign change knots in $(0,\alpha)$. Since $T_{6l}(ze^{\frac{2\pi
i}{3}})=e^{\frac{2\pi i}{3}}\,T_{6l}(z)$, we see that any zero of
$T_{6l}$ in $(0,\infty)$ must be simple, otherwise $T_{6l}$ would
have at least $6l+3$ zeros, contradicting $\deg(T_{6l})\leq 6l+1$.
Similarly, using (\ref{ortT6l1})--(\ref{ortT6l2}) we show that the
polynomials $T_{6l+j}, 1\leq j\leq 5$, have only simple zeros in
$(0,\infty)$.

Now we prove that the functions $Y_{n}$ have only simple zeros in
$(-\infty,0)$. We know that $Y_{6l}$ has at least $l$ sign change
knots in $(-\infty,0)$. If we assume that $Y_{6l}$ has a zero of
multiplicity $\geq 2$, then $\deg{R_{6l}}\geq 3l+6$, and so we
would have
\[
Y_{6l}(z)/R_{6l}(z)=O(1/z^{6l+6}),\qquad z\rightarrow\infty.
\]
Reasoning as above, we arrive to the fact that $\deg{T_{6l}}\geq
6l+3$, which is impossible. Similarly we see that the zeros of
$Y_{6l+j}, 1\leq j\leq 5,$ contained in $(-\infty,0)$, must be
simple.
\end{proof}

\noindent{\bf Proof of Theorem \ref{theointerlacing}.} Let
$x\in(0,\alpha)$ and assume that $Q_{n}(x)=Q_{n+1}(x)=0$. Take
$A=1$, $B=-xQ_{n}'(x)/Q_{n+1}'(x)$. For this choice of $A$ and
$B$, the polynomial $T_{n}$ defined by (\ref{defnTn}) satisfies
$T_{n}(x)=T_{n}'(x)=0$, contradicting Proposition
\ref{propsimplezeroscomb}.

Let $x\in(0,\alpha)$ be arbitrary but fixed. Take now
$A=Q_{n+1}(x)/x$ and $B=-Q_{n}(x)$. For this choice of $A$ and
$B$, we have $T_{n}(x)=0$, therefore $T_{n}'(x)\neq 0$, or
equivalently
\[
L_{n}(x):=\frac{Q_{n+1}(x)\,Q_{n}(x)}{x}
+Q_{n+1}(x)\,Q_{n}'(x)-Q_{n}(x)\,Q_{n+1}'(x)\neq 0.
\]
In particular, the sign of $L_{n}$ is constant on $(0,\alpha)$.
Evaluating $L_{n}$ at two consecutive zeros of $Q_{n}$ ($Q_{n+1}$)
on $(0,\alpha)$, we see immediately that there must be an
intermediate zero of $Q_{n+1}$ ($Q_{n}$).

The same argument proves the interlacing property of the zeros of
$\Psi_{n}$ and $\Psi_{n+1}$. \hfill $\Box$

\begin{prop}\label{interlacingzerosQn}
Let the roots of the polynomials $Q_{3k+i}$, $0\leq i\leq 2$, in
the interval $(0,\alpha)$, be defined as follows:
\[
x_{1}^{(3k+i)}<x_{2}^{(3k+i)}<x_{3}^{(3k+i)}<\cdots<x_{k-1}^{(3k+i)}<x_{k}^{(3k+i)}.
\]
Then
\begin{equation}\label{eq:interl1}
x_{1}^{(3k)}<x_{1}^{(3k+1)}<x_{2}^{(3k)}<x_{2}^{(3k+1)}<\cdots<x_{k}^{(3k)}<x_{k}^{(3k+1)},
\end{equation}
\begin{equation}\label{eq:interl2}
x_{1}^{(3k+1)}<x_{1}^{(3k+2)}<x_{2}^{(3k+1)}<x_{2}^{(3k+2)}<\cdots<x_{k}^{(3k+1)}<x_{k}^{(3k+2)},
\end{equation}
\begin{equation}\label{eq:interl3}
x_{1}^{(3k+3)}<x_{1}^{(3k+2)}<x_{2}^{(3k+3)}<x_{2}^{(3k+2)}<\cdots<x_{k}^{(3k+2)}<x_{k+1}^{(3k+3)}.
\end{equation}
\end{prop}
\begin{proof}
If we write
\[
Q_{3k-2}(z)=b_{1}^{(3k-2)}z+\cdots+z^{3k-2},\qquad
Q_{3k}(z)=b_{0}^{(3k)}+\cdots+z^{3k},\qquad
Q_{3k+1}(z)=b_{1}^{(3k+1)}z+\cdots+z^{3k+1},
\]
from (\ref{recurrence}) we obtain the relation
$b_{0}^{(3k)}-b_{1}^{(3k+1)}=a_{3k}\,b_{1}^{(3k-2)}$. Vieta
formulas show that
\[
b_{0}^{(3k)}=(-1)^{3k}(x_{1}^{(3k)}\cdots x_{k}^{(3k)})^{3},\qquad
b_{1}^{(3k+1)}=(-1)^{3k}(x_{1}^{(3k+1)}\cdots x_{k}^{(3k+1)})^{3},
\]
and similarly $b_{1}^{(3k-2)}$ equals $(-1)^{3k-1}$ times the
product of all non-zero roots of $Q_{3k-2}$. Since $a_{3k}>0$ and
the product of all non-zero roots of $Q_{3k-2}$ is also positive,
we deduce that $(x_{1}^{(3k)}\cdots
x_{k}^{(3k)})^{3}<(x_{1}^{(3k+1)}\cdots x_{k}^{(3k+1)})^{3}$. This
inequality and Theorem \ref{theointerlacing} imply
(\ref{eq:interl1}). Similarly we show that $(x_{1}^{(3k+1)}\cdots
x_{k}^{(3k+1)})^{3}<(x_{1}^{(3k+2)}\cdots x_{k}^{(3k+2)})^{3}$,
which implies (\ref{eq:interl2}). Finally, (\ref{eq:interl3})
follows directly from Theorem \ref{theointerlacing}. \end{proof}


\section{Ratio asymptotics of the polynomials $Q_{n}$ and $Q_{n,2}$}
\label{sectionratio}

Let
\begin{equation}\label{defnHn1}
H_{n}:=\frac{Q_{n}\Psi_{n}}{Q_{n,2}}.
\end{equation}
Notice that $H_{n}$ is real-valued on $(-\infty,0)$ and has
constant sign on this interval. Having in mind the definitions
$(\ref{defn:P3k})$--$(\ref{defn:Pn2})$, we have:

\begin{prop}\label{prop3}
Let $l\geq 0$ be an arbitrary integer. Then the following
orthogonality conditions hold:
\begin{equation}\label{eqortQn20}
0=\int_{-b^{3}}^{-a^{3}}\tau^{k}\,P_{6l+j,2}(\tau)
\frac{|H_{6l+j}(\sqrt[3]{\tau})|\,s_{2}(\sqrt[3]{\tau})}
{|\sqrt[3]{\tau}\,\,P_{6l+j}(\tau)|}\,d\tau,\qquad 0\leq k\leq
l-1,\quad j=0,3,
\end{equation}
\begin{equation}\label{eqortQn22}
0=\int_{-b^{3}}^{-a^{3}}\tau^{k}\,P_{6l+2+j,2}(\tau)
\frac{|H_{6l+2+j}(\sqrt[3]{\tau})|\,s_{2}(\sqrt[3]{\tau})}{|\tau^{2/3}
P_{6l+2+j}(\sqrt[3]{\tau})|}\,d\tau,\qquad 0\leq k\leq l-1,\quad
j=0,3,
\end{equation}
\begin{equation}\label{eqortQn21}
0=\int_{-b^{3}}^{-a^{3}}\tau^{k}\,P_{6l+1,2}(\tau)
\frac{|H_{6l+1}(\sqrt[3]{\tau})|\,s_{2}(\sqrt[3]{\tau})}{|\tau
P_{6l+1}(\tau)|}\,d\tau,\qquad 0\leq k\leq l-1,
\end{equation}
\begin{equation}\label{eqortQn24}
0=\int_{-b^{3}}^{-a^{3}}\tau^{k}\,P_{6l+4,2}(\tau)
\frac{|H_{6l+4}(\sqrt[3]{\tau})|\,s_{2}(\sqrt[3]{\tau})} {|\tau
P_{6l+4}(\tau)|}\,d\tau,\qquad 0\leq k\leq l.
\end{equation}
\end{prop}
\begin{proof}
This orthogonality conditions follow immediately from
$(\ref{eq:ortintPsi6l})$--$(\ref{eq:ortintPsi6l4})$.
\end{proof}

\begin{prop}\label{prop4}
Let $k\geq 0$ be an arbitrary integer. Then the following
orthogonality conditions hold:
\begin{equation}\label{eq:ortQ3kvar}
0=\int_{0}^{\alpha^{3}}\tau^{j}\,P_{3k}(\tau)
\frac{s_{1}(\sqrt[3]{\tau})}{P_{3k,2}(\tau)}\frac{d\tau}{\tau^{2/3}},\qquad
0\leq j\leq k-1.
\end{equation}
\begin{equation}\label{eq:ortQ3k1var}
0=\int_{0}^{\alpha^{3}}\tau^{j}\,P_{3k+1}(\tau)
\frac{s_{1}(\sqrt[3]{\tau})}{P_{3k+1,2}(\tau)}\sqrt[3]{\tau}\,d\tau,\qquad
0\leq j\leq k-1.
\end{equation}
\begin{equation}\label{eq:ortQ3k2var}
0=\int_{0}^{\alpha^{3}}\tau^{j}\,P_{3k+2}(\tau)
\frac{s_{1}(\sqrt[3]{\tau})}{P_{3k+2,2}(\tau)}\sqrt[3]{\tau}\,d\tau,\qquad
0\leq j\leq k-1.
\end{equation}
\end{prop}
\begin{proof}
This orthogonality conditions follow immediately from
$(\ref{eq:ortQ6lvar})$--$(\ref{eq:ortQ6l2var})$.
\end{proof}

Observe that by Proposition \ref{propzerosPsi1}, for each
$j\in\{0,1,2,3,5\}$, $P_{6l+j,2}$ is a polynomial of degree $l$,
and $P_{6l+4,2}$ has degree $l+1$. By Proposition
\ref{propmaximal}, for each $k\geq 0$ and $j\in\{0,1,2\}$,
$P_{3k+j}$ has degree $k$.

For each integer $j\geq 0$ we let
\[
K_{3j}:=\Big(\int_{0}^{\alpha^{3}}P_{3j}^{2}(\tau)\frac{s_{1}(\sqrt[3]{\tau})}{P_{3j,2}(\tau)}
\frac{d\tau}{\tau^{2/3}}\Big)^{-1/2},
\]
\[
K_{3j+1}:=\Big(\int_{0}^{\alpha^{3}}P_{3j+1}^{2}(\tau)\frac{s_{1}(\sqrt[3]{\tau})\sqrt[3]{\tau}}{P_{3j+1,2}(\tau)}
\,d\tau\Big)^{-1/2},
\]
\[
K_{3j+2}:=\Big(\int_{0}^{\alpha^{3}}P_{3j+2}^{2}(\tau)\frac{s_{1}(\sqrt[3]{\tau})\sqrt[3]{\tau}}{P_{3j+2,2}(\tau)}
\,d\tau\Big)^{-1/2}.
\]
Similarly, we define for each integer $j\geq 0$ the following
constants:
\[
K_{3j,2}:=\Big(\int_{-b^{3}}^{-a^{3}}P_{3j,2}^{2}(\tau)\,\frac{|H_{3j}(\sqrt[3]{\tau})|}
{|\sqrt[3]{\tau}\,P_{3j}(\tau)|}\,s_{2}(\sqrt[3]{\tau})\,d\tau\Big)^{-1/2},
\]
\[
K_{3j+1,2}:=\Big(\int_{-b^{3}}^{-a^{3}}P_{3j+1,2}^{2}(\tau)\,\frac{|H_{3j+1}(\sqrt[3]{\tau})|}
{|\tau P_{3j+1}(\tau)|}\,s_{2}(\sqrt[3]{\tau})\,d\tau\Big)^{-1/2},
\]
\[
K_{3j+2,2}:=\Big(\int_{-b^{3}}^{-a^{3}}P_{3j+2,2}^{2}(\tau)\,\frac{|H_{3j+2}(\sqrt[3]{\tau})|}
{|\tau^{2/3}
P_{3j+2}(\tau)|}\,s_{2}(\sqrt[3]{\tau})\,d\tau\Big)^{-1/2}.
\]
We need to introduce more notations. Let
\begin{equation}\label{defnkappankappan2}
\kappa_{n}:=K_{n},\qquad \kappa_{n,2}:=\frac{K_{n,2}}{K_{n}},
\end{equation}
consider the polynomials
\begin{equation}\label{eq:definitionpn}
p_{n}:=\kappa_{n}\,P_{n},\qquad p_{n,2}:=\kappa_{n,2}\,P_{n,2},
\end{equation}
and the functions
\begin{equation}\label{eq:definitionhn}
h_{n}:=K_{n}^{2}\,H_{n}.
\end{equation}
Finally, we introduce the following positive varying measures:
\begin{equation}\label{orthogvaryingmeasures}
d\nu_{3j}(\tau):=\frac{s_{1}(\sqrt[3]{\tau})}{P_{3j,2}(\tau)}\frac{d\tau}{\tau^{2/3}},
\quad\,\,
d\nu_{3j+1}(\tau):=\frac{s_{1}(\sqrt[3]{\tau})\sqrt[3]{\tau}}{P_{3j+1,2}(\tau)}\,d\tau,
\quad\,\,
d\nu_{3j+2}(\tau):=\frac{s_{1}(\sqrt[3]{\tau})\sqrt[3]{\tau}}{P_{3j+2,2}(\tau)}\,d\tau,
\end{equation}
\[
d\nu_{3j,2}(\tau):=\frac{|h_{3j}(\sqrt[3]{\tau})|}
{|\sqrt[3]{\tau}\,P_{3j}(\tau)|}\,s_{2}(\sqrt[3]{\tau})\,d\tau,
\quad\,\, d\nu_{3j+1,2}(\tau):=\frac{|h_{3j+1}(\sqrt[3]{\tau})|}
{|\tau P_{3j+1}(\tau)|}\,s_{2}(\sqrt[3]{\tau})\,d\tau,\quad\,\,
d\nu_{3j+2,2}(\tau):=\frac{|h_{3j+2}(\sqrt[3]{\tau})|}
{|\tau^{2/3} P_{3j+2}(\tau)|}\,s_{2}(\sqrt[3]{\tau})\,d\tau.
\]

\begin{prop}\label{orthonormalitypn}
The polynomials $p_{n}$ and $p_{n,2}$ are orthonormal polynomials
with respect to the measures $d\nu_{n}$ and $d\nu_{n,2}$,
respectively. This is, for every $n\geq 0$,
\[
1=\int_{0}^{\alpha^{3}}p_{n}^{2}(\tau)\,d\nu_{n}(\tau),\qquad
1=\int_{-b^{3}}^{-a^{3}}p_{n,2}^{2}(\tau)\,d\nu_{n,2}(\tau),
\]
\[
0=\int_{0}^{\alpha^{3}}\tau^{j}p_{n}(\tau)\,d\nu_{n}(\tau),\quad
\mbox{for all}\quad j<\deg{p_{n}},\qquad\qquad
0=\int_{-b^{3}}^{-a^{3}}\tau^{j}p_{n,2}(\tau)\,d\nu_{n,2}(\tau),\quad
\mbox{for all}\quad j<\deg{p_{n,2}}.
\]
\end{prop}
\begin{proof}
It follows immediately from Propositions \ref{prop3} and
\ref{prop4}.
\end{proof}

Using $(\ref{eq:rep4})$--$(\ref{eq:rep6})$, it is easy to check
that the functions $h_{n}$ have the following representations:
\begin{equation}\label{intreph3k}
h_{3k}(z)=z^{2}\int_{0}^{\alpha^{3}}\frac{p_{3k}^{2}(\tau)}{\tau-z^{3}}\,d\nu_{3k}(\tau),\qquad\,
h_{3k+1}(z)=z\int_{0}^{\alpha^{3}}\frac{p_{3k+1}^{2}(\tau)}{\tau-z^{3}}\,d\nu_{3k+1}(\tau),\qquad\,
h_{3k+2}(z)=z^{3}\int_{0}^{\alpha^{3}}\frac{p_{3k+2}^{2}(\tau)}{\tau-z^{3}}\,d\nu_{3k+2}(\tau).
\end{equation}

\begin{lem}\label{lemmaweakstar}
Assume that $s_{1}>0$ a.e. on $[0,\alpha]$, and $s_{2}>0$ a.e. on
$[-b,-a]$. Then
\begin{equation}\label{weakstarasymp1}
p_{n}^{2}(\tau)\,d\nu_{n}(\tau)\stackrel{*}{\longrightarrow}
\frac{1}{\pi}\frac{d\tau}{\sqrt{(\alpha^{3}-\tau)\tau}},\quad
\tau\in[0,\alpha^3],
\end{equation}
\begin{equation}\label{weakstarasymp2}
p_{n,2}^{2}(\tau)\,d\nu_{n,2}(\tau)\stackrel{*}{\longrightarrow}
\frac{1}{\pi}\frac{d\tau}{\sqrt{(-a^{3}-\tau)(\tau+b^{3})}},\quad
\tau\in[-b^3,-a^3].
\end{equation}
Consequently, the following limits hold uniformly on closed
subsets of $\overline{\mathbb{C}}\setminus S_{0}:$
\begin{equation}\label{asymph3k}
\lim_{k\rightarrow\infty}h_{3k}(z)=-\frac{z^2}{\sqrt{(z^{3}-\alpha^{3})z^{3}}},
\end{equation}
\begin{equation}\label{asymph3k1}
\lim_{k\rightarrow\infty}h_{3k+1}(z)=-\frac{z}{\sqrt{(z^{3}-\alpha^{3})z^{3}}},
\end{equation}
\begin{equation}\label{asymph3k2}
\lim_{k\rightarrow\infty}h_{3k+2}(z)=-\frac{z^{3}}{\sqrt{(z^{3}-\alpha^{3})z^{3}}},
\end{equation}
where the branch of the square root is taken so that $\sqrt{x}>0$
for $x>0$.
\end{lem}
\begin{proof}
Let us define the measures
\[
d\mu_{3k}(\tau)=\frac{s_{1}(\sqrt[3]{\tau})}{\tau^{2/3}}\,d\tau,\qquad\qquad
d\mu_{3k+1}(\tau)=d\mu_{3k+2}(\tau)=s_{1}(\sqrt[3]{\tau})\sqrt[3]{\tau}\,d\tau.
\]
According to \cite[Definition 2]{BCL}, for each $i\in\{0,1,2\}$
and $k\in\mathbb{Z}$, the system $(\{d\mu_{3l+i}\},
\{P_{3l+i,2}\}, k)_{l\geq 1}$ is strongly admissible on
$[0,\alpha^3]$. So by \cite[Corollary 3]{BCL},
\[
\lim_{l\rightarrow\infty}\int_{0}^{\alpha^{3}}f(\tau)\,p_{3l+i}^{2}(\tau)\,
\frac{d\mu_{3l+i}(\tau)}{P_{3l+i,2}(\tau)}=\frac{1}{\pi}\int_{0}^{\alpha^{3}}f(\tau)
\frac{d\tau}{\sqrt{(\alpha^{3}-\tau)\tau}},
\]
for every $f$ continuous on $[0,\alpha^{3}]$. Since
$d\nu_{3l+i}(\tau)=d\mu_{3l+i}(\tau)/P_{3l+i,2}(\tau)$,
(\ref{weakstarasymp1}) follows. The formulas
(\ref{asymph3k})--(\ref{asymph3k2}) are a consequence of
(\ref{weakstarasymp1}) and (\ref{intreph3k}).

Similarly, if we define the measures
\[
d\lambda_{3k}(\tau)=\frac{|h_{3k}(\sqrt[3]{\tau})|}
{|\sqrt[3]{\tau}|}\,s_{2}(\sqrt[3]{\tau})\,d\tau,\qquad
d\lambda_{3k+1}(\tau)=\frac{|h_{3k+1}(\sqrt[3]{\tau})|}
{|\tau|}\,s_{2}(\sqrt[3]{\tau})\,d\tau, \qquad
d\lambda_{3k+2}(\tau)=\frac{|h_{3k+2}(\sqrt[3]{\tau})|}
{|\tau^{2/3}|}\,s_{2}(\sqrt[3]{\tau})\,d\tau,
\]
then for each $i\in\{0,1,2\}$ and each $k\in\mathbb{Z}$, the
system $(\{d\lambda_{3l+i}\},\{|P_{3l+i}|\},k)$ is strongly
admissible on $[-b^{3},-a^{3}]$, and (\ref{weakstarasymp2})
follows as before.
\end{proof}

For each $i\in\{0,\ldots,5\}$, we consider the families of
rational functions
\begin{equation}\label{familiesrational}
\Big\{\frac{P_{6k+i+1}(z)}{P_{6k+i}(z)}\Big\}_{k},\qquad
\Big\{\frac{P_{6k+i+1,2}(z)}{P_{6k+i,2}(z)}\Big\}_{k}.
\end{equation}
By Theorem \ref{theointerlacing}, these families are uniformly
bounded on compact subsets of $\mathbb{C}\setminus[0,\alpha^{3}]$
and $\mathbb{C}\setminus[-b^3,-a^3]$, respectively. Therefore, by
Montel's theorem there exists a sequence of integers
$\Lambda\subset\mathbb{N}$ so that for each $i\in\{0,\ldots,5\}$,
\begin{equation}\label{defntildeF1i}
\lim_{k\in\Lambda}\frac{P_{6k+i+1}(z)}{P_{6k+i}(z)}=\widetilde{F}_{1}^{(i)}(z),\qquad
z\in\mathbb{C}\setminus[0,\alpha^{3}],
\end{equation}
\begin{equation}\label{defntildeF2i}
\lim_{k\in\Lambda}\frac{P_{6k+i+1,2}(z)}{P_{6k+i,2}(z)}=\widetilde{F}_{2}^{(i)}(z),\qquad
z\in\mathbb{C}\setminus[-a^{3},-b^{3}],
\end{equation}
where the limits hold uniformly on compact subsets of the
indicated regions. Our goal is to show that we obtain the same
limiting functions $\widetilde{F}_{j}^{(i)}$, no matter which
convergent subsequences we take.

Taking into account the degree of $P_{n}$ and $P_{n,2}$, from
(\ref{defntildeF1i})--(\ref{defntildeF2i}) we deduce:
$\widetilde{F}_{1}^{(i)}$ and $1/\widetilde{F}_{1}^{(i)}$ are
analytic in $\mathbb{C}\setminus[0,\alpha^{3}]$,
$\widetilde{F}_{2}^{(i)}$ and $1/\widetilde{F}_{2}^{(i)}$ are
analytic in $\mathbb{C}\setminus[-b^3,-a^{3}]$, and as
$z\rightarrow\infty$,
\begin{equation}\label{LaurentexpansionFinfinity}
\left\{
\begin{array}{ll}

\widetilde{F}_{1}^{(i)}(z)=1+O(1/z), & i\in\{0,1,3,4\},\\

\widetilde{F}_{1}^{(i)}(z)=z+O(1), & i\in\{2,5\},\\

\widetilde{F}_{2}^{(i)}(z)=1+O(1/z), & i\in\{0,1,2\},\\

\widetilde{F}_{2}^{(i)}(z)=z+O(1), & i\in\{3,5\},\\

\widetilde{F}_{2}^{(4)}(z)=1/z+O(1/z^2).
\end{array}\right.
\end{equation}

Given a Borel measurable function $w\geq 0$ defined on the
interval $[c,d]$ that satisfies the Szeg\H{o} condition
\[
\frac{\log w(t)}{\sqrt{(d-t)(t-c)}}\in L^{1}(dt),
\]
let
\[
S(w;z):=\exp\Big\{
\frac{d-c}{4\pi}\sqrt{\Big(\frac{2z-c-d}{d-c}\Big)^2-1}\int_{c}^{d}\frac{\log
w(t)}{t-z}\frac{dt}{\sqrt{(d-t)(t-c)}} \Big\}
\]
denote the Szeg\H{o} function on
$\overline{\mathbb{C}}\setminus[c,d]$ associated with $w$ (see
\cite{Szego}). In particular, if $w$ is continuous  at $x\in[c,d]$
and $w(x)>0$, then the limit
\begin{equation}\label{boundaryvalueSzegofunction}
\lim_{z\rightarrow x}|S(w;z)|^{2}=\frac{1}{w(x)}
\end{equation}
holds. We will indicate this below by writing
$|S(w;x)|^{2}\,w(x)=1$.

Throughout this section we are always assuming that $s_{1}>0$ a.e.
on $[0,\alpha]$, and $s_{2}>0$ a.e. on $[-b,-a]$. If $f_{n}\in
H(\Omega)$, $\Omega\subset\overline{\mathbb{C}}$, the notation
\[
\lim_{n\in\widetilde{\Lambda}} f_{n}(z)=F(z),\qquad
z\in\Omega,\quad \widetilde{\Lambda}\subset\mathbb{N},
\]
stands for the uniform convergence of $f_{n}$ to $F$ on each
compact subset of $\Omega$.

By $(\ref{eq:ortQ3kvar})$--$(\ref{eq:ortQ3k1var})$ we have:
\[
0=\int_{0}^{\alpha^{3}}\tau^{j}P_{6k}(\tau)\,d\nu_{6k}(\tau),\qquad
0\leq j\leq 2k-1,
\]
\[
0=\int_{0}^{\alpha^{3}}\tau^{j}P_{6k+1}(\tau)\,g_{6k}(\tau)\,d\nu_{6k}(\tau),\qquad
0\leq j\leq 2k-1,
\]
where $g_{6k}(\tau):=\tau P_{6k,2}(\tau)/P_{6k+1,2}(\tau)$. Using
(\ref{defntildeF2i}),
\[
\lim_{k\in\Lambda}g_{6k}(\tau)=\frac{\tau}{\widetilde{F}_{2}^{(0)}(\tau)},\qquad\mbox{uniformly
on}\quad[0,\alpha^3].
\]
Since $\deg(P_{6k})=\deg(P_{6k+1})$, applying \cite[Theorem
2]{BCL} (result on relative asymptotics of polynomials orthogonal
with respect to varying measures), we obtain
\begin{equation}\label{tildeF10}
\lim_{k\in\Lambda}\frac{P_{6k+1}(z)}{P_{6k}(z)}
=\frac{S_{1}^{(0)}(z)}{S_{1}^{(0)}(\infty)}=\widetilde{F}_{1}^{(0)}(z),\qquad
z\in\overline{\mathbb{C}}\setminus[0,\alpha^{3}],
\end{equation}
where $S_{1}^{(0)}$ is the Szeg\H{o} function on
$\overline{\mathbb{C}}\setminus[0,\alpha^{3}]$ associated with the
weight $\tau/\widetilde{F}_{2}^{(0)}(\tau)$,
$\tau\in[0,\alpha^{3}]$.

By Proposition \ref{prop4} we have:
\[
0=\int_{0}^{\alpha^{3}}\tau^{j}P_{6k+2}(\tau)\,d\nu_{6k+2}(\tau),\qquad
0\leq j\leq 2k-1,
\]
\[
0=\int_{0}^{\alpha^{3}}\tau^{j}P_{6k+3}(\tau)\,g_{6k+2}(\tau)\,d\nu_{6k+2}(\tau),\qquad
0\leq j\leq 2k,
\]
where $g_{6k+2}(\tau):=P_{6k+2,2}(\tau)/(\tau P_{6k+3,2}(\tau))$.
Let $P_{6k+2}^{*}$ be the monic polynomial of degree $2k$
orthogonal with respect to the measure
$d\nu_{6k+3}(\tau)=g_{6k+2}(\tau)\,d\nu_{6k+2}(\tau)$. Since
$\deg(P_{6k+2}^{*})=\deg(P_{6k+2})$, again by \cite[Theorem
2]{BCL} we obtain
\[
\lim_{k\in\Lambda}\frac{P_{6k+2}^{*}(z)}{P_{6k+2}(z)}
=\frac{S_{1}^{(2)}(z)}{S_{1}^{(2)}(\infty)},\qquad
z\in\overline{\mathbb{C}}\setminus[0,\alpha^{3}],
\]
where $S_{1}^{(2)}$ is the Szeg\H{o} function on
$\overline{\mathbb{C}}\setminus[0,\alpha^{3}]$ with respect to the
weight $1/(\tau \widetilde{F}_{2}^{(2)}(\tau))$.

Let $\phi_{1}$ denote the conformal mapping that maps
$\overline{\mathbb{C}}\setminus[0,\alpha^{3}]$ onto the exterior
of the unit circle and satisfies $\phi_{1}(\infty)=\infty$ and
$\phi_{1}'(\infty)>0$. Then, by \cite[Theorem 1]{BCL} (result on
ratio asymptotics of polynomials orthogonal with respect to
varying measures) we have
\[
\lim_{k\in\Lambda}\frac{P_{6k+3}(z)}{P_{6k+2}^{*}(z)}
=\frac{\phi_{1}(z)}{\phi_{1}'(\infty)},\qquad
z\in\mathbb{C}\setminus[0,\alpha^{3}].
\]
Therefore, we conclude that
\begin{equation}\label{tildeF12}
\lim_{k\in\Lambda}\frac{P_{3k+3}(z)}{P_{6k+2}(z)}
=\frac{S_{1}^{(2)}(z)}{S_{1}^{(2)}(\infty)}
\,\frac{\phi_{1}(z)}{\phi_{1}'(\infty)}=\widetilde{F}_{1}^{(2)}(z),\qquad
z\in\mathbb{C}\setminus[0,\alpha^{3}].
\end{equation}

The same arguments used before show that
\begin{equation}\label{tildeF11}
\lim_{k\in\Lambda}\frac{P_{6k+i+1}(z)}{P_{6k+i}(z)}=\frac{S_{1}^{(i)}(z)}{S_{1}^{(i)}(\infty)}
=\widetilde{F}_{1}^{(i)}(z),\qquad
z\in\overline{\mathbb{C}}\setminus[0,\alpha^{3}],\quad
i\in\{1,3,4\},
\end{equation}
\begin{equation}\label{tildeF15}
\lim_{k\in\Lambda}\frac{P_{6k+6}(z)}{P_{6k+5}(z)}
=\frac{S_{1}^{(5)}(z)}{S_{1}^{(5)}(\infty)}\,\frac{\phi_{1}(z)}{\phi_{1}'(\infty)}=\widetilde{F}_{1}^{(5)}(z),
\qquad z\in\mathbb{C}\setminus[0,\alpha^{3}],
\end{equation}
where $S_{1}^{(1)}, S_{1}^{(3)}, S_{1}^{(4)}$, and $S_{1}^{(5)}$
are the Szeg\H{o} functions on
$\overline{\mathbb{C}}\setminus[0,\alpha^{3}]$ with respect to the
weights $1/\widetilde{F}_{2}^{(1)}(\tau)$,
$\tau/\widetilde{F}_{2}^{(3)}(\tau),\linebreak
1/\widetilde{F}_{2}^{(4)}(\tau)$, and $1/(\tau
\widetilde{F}_{2}^{(5)}(\tau))$, respectively.

Applying now the orthogonality conditions from Proposition
\ref{prop3} and (\ref{asymph3k})--(\ref{asymph3k2}), we deduce:
\begin{equation}\label{tildeF20}
\lim_{k\in\Lambda}\frac{P_{6k+i+1,2}(z)}{P_{6k+i,2}(z)}
=\frac{S_{2}^{(i)}(z)}{S_{2}^{(i)}(\infty)}
=\widetilde{F}_{2}^{(i)}(z),\qquad
z\in\overline{\mathbb{C}}\setminus[-b^{3},-a^{3}],\quad
i\in\{0,1,2\},
\end{equation}
\begin{equation}\label{tildeF23}
\lim_{k\in\Lambda}\frac{P_{6k+i+1,2}(z)}{P_{6k+i,2}(z)}
=\frac{S_{2}^{(i)}(z)}{S_{2}^{(i)}(\infty)}\,\frac{\phi_{2}(z)}{\phi_{2}'(\infty)}
=\widetilde{F}_{2}^{(i)}(z),\qquad
z\in\mathbb{C}\setminus[-b^{3},-a^{3}],\quad i\in\{3,5\},
\end{equation}
\begin{equation}\label{tildeF24}
\lim_{k\in\Lambda}\frac{P_{6k+5,2}(z)}{P_{6k+4,2}(z)}=\frac{S_{2}^{(4)}(\infty)}{S_{2}^{(4)}(z)}
\,\frac{\phi_{2}'(\infty)}{\phi_{2}(z)}
=\widetilde{F}_{2}^{(4)}(z),\qquad
z\in\mathbb{C}\setminus[-b^{3},-a^{3}],
\end{equation}
where $S_{2}^{(0)},\ldots,S_{2}^{(5)},$ are the Szeg\H{o}
functions on $\overline{\mathbb{C}}\setminus[-b^{3},-a^{3}]$
associated with the weights
$$\frac{1}{|\tau\widetilde{F}_{1}^{(0)}(\tau)|},\quad
\frac{|\tau|}{|\widetilde{F}_{1}^{(1)}(\tau)|},\quad
\frac{1}{|\widetilde{F}_{1}^{(2)}(\tau)|}, \quad
\frac{1}{|\tau\widetilde{F}_{1}^{(3)}(\tau)|},\quad
\frac{|\widetilde{F}_{1}^{(4)}(\tau)|}{|\tau|},\quad
\frac{1}{|\widetilde{F}_{1}^{(5)}(\tau)|},
$$
respectively, and $\phi_{2}$ is the conformal mapping that maps
$\overline{\mathbb{C}}\setminus[-b^{3},-a^{3}]$ onto the exterior
of the unit circle that satisfies the conditions
$\phi_{2}(\infty)=\infty$ and $\phi_{2}'(\infty)>0$.

\begin{prop}\label{prop5}
There exist positive constants $c_{k}^{(l)}$ so that the functions
$F_{k}^{(l)}:=c_{k}^{(l)}\widetilde{F}_{k}^{(l)}$ satisfy the
following boundary value conditions:
\begin{equation}\label{eq:boundaryvalue1}
|F_{1}^{(l)}(\tau)|^{2}\,\frac{\tau}{F_{2}^{(l)}(\tau)}=1,\quad
\tau\in(0,\alpha^{3}],\quad l=0,3,
\end{equation}
\begin{equation}\label{eq:boundaryvalue2}
|F_{1}^{(l)}(\tau)|^{2}\,\frac{1}{F_{2}^{(l)}(\tau)}=1,\quad
\tau\in[0,\alpha^{3}],\quad l=1,4,
\end{equation}
\begin{equation}\label{eq:boundaryvalue3}
|F_{1}^{(l)}(\tau)|^{2}\,\frac{1}{\tau\,F_{2}^{(l)}(\tau)}=1,\quad
\tau\in(0,\alpha^{3}],\quad l=2,5,
\end{equation}
\begin{equation}\label{eq:boundaryvalue4}
|F_{2}^{(l)}(\tau)|^{2}\,\frac{1}{|\tau\,F_{1}^{(l)}(\tau)|}=1,\quad
\tau\in[-b^{3},-a^{3}],\quad l=0,3,
\end{equation}
\begin{equation}\label{eq:boundaryvalue5}
|F_{2}^{(l)}(\tau)|^{2}\,\frac{|\tau|}{|F_{1}^{(l)}(\tau)|}=1,\quad
\tau\in[-b^{3},-a^{3}],\quad l=1,4,
\end{equation}
\begin{equation}\label{eq:boundaryvalue6}
|F_{2}^{(l)}(\tau)|^{2}\,\frac{1}{|F_{1}^{(l)}(\tau)|}=1,\quad
\tau\in[-b^{3},-a^{3}],\quad l=2,5.
\end{equation}
\end{prop}
\begin{proof}
It follows from the relations (\ref{tildeF10})--(\ref{tildeF24}),
the definition of the Szeg\H{o} functions $S_{j}^{(i)}$ and
(\ref{boundaryvalueSzegofunction}), that there exist positive
constants $\omega_{1}^{(l)}, \omega_{2}^{(l)}$, such that
\begin{equation}\label{eq:boundaryvaluetilde1}
|\widetilde{F}_{1}^{(l)}(\tau)|^{2}\,\frac{\tau}{\widetilde{F}_{2}^{(l)}(\tau)}
=\frac{1}{\omega_{1}^{(l)}},\quad \tau\in(0,\alpha^{3}],\quad
l=0,3,
\end{equation}
\begin{equation}\label{eq:boundaryvaluetilde2}
|\widetilde{F}_{1}^{(l)}(\tau)|^{2}\,\frac{1}{\widetilde{F}_{2}^{(l)}(\tau)}=\frac{1}{\omega_{1}^{(l)}},\quad
\tau\in[0,\alpha^{3}],\quad l=1,4,
\end{equation}
\begin{equation}\label{eq:boundaryvaluetilde3}
|\widetilde{F}_{1}^{(l)}(\tau)|^{2}\,\frac{1}{\tau\,\widetilde{F}_{2}^{(l)}(\tau)}=\frac{1}{\omega_{1}^{(l)}},\quad
\tau\in(0,\alpha^{3}],\quad l=2,5,
\end{equation}
\begin{equation}\label{eq:boundaryvaluetilde4}
|\widetilde{F}_{2}^{(l)}(\tau)|^{2}\,\frac{1}{|\tau\,\widetilde{F}_{1}^{(l)}(\tau)|}
=\frac{1}{\omega_{2}^{(l)}},\quad \tau\in[-b^{3},-a^{3}],\quad
l=0,3,
\end{equation}
\begin{equation}\label{eq:boundaryvaluetilde5}
|\widetilde{F}_{2}^{(l)}(\tau)|^{2}\,\frac{|\tau|}{|\widetilde{F}_{1}^{(l)}(\tau)|}
=\frac{1}{\omega_{2}^{(l)}},\quad \tau\in[-b^{3},-a^{3}],\quad
l=1,4,
\end{equation}
\begin{equation}\label{eq:boundaryvaluetilde6}
|\widetilde{F}_{2}^{(l)}(\tau)|^{2}\,\frac{1}{|\widetilde{F}_{1}^{(l)}(\tau)|}
=\frac{1}{\omega_{2}^{(l)}},\quad \tau\in[-b^{3},-a^{3}],\quad
l=2,5,
\end{equation}
where
\begin{equation}\label{defnomega1l}
\omega_{1}^{(l)}=(S_{1}^{(l)}(\infty))^{2},\quad \mbox{for}\quad
l=0,1,3,4,
\end{equation}
\begin{equation}\label{defnomega1lmas}
\omega_{1}^{(l)}=(S_{1}^{(l)}(\infty)\,\phi_{1}'(\infty))^{2},\quad
\mbox{for}\quad l=2,5,
\end{equation}
\begin{equation}\label{defnomega2l}
\omega_{2}^{(l)}=(S_{2}^{(l)}(\infty))^{2},\quad\mbox{for}\quad
l=0,1,2,
\end{equation}
\begin{equation}\label{defnomega2lmas}
\omega_{2}^{(l)}=(S_{2}^{(l)}(\infty)\,\phi_{2}'(\infty))^{2},\quad\mbox{for}\quad
l=3,5,
\end{equation}
\begin{equation}\label{defnomega2lmasmas}
\omega_{2}^{(4)}=1/(S_{2}^{(4)}(\infty)\,\phi_{2}'(\infty))^{2}.
\end{equation}
The positive constants $c_{k}^{(l)}$ that satisfy the requirements
are $c_{1}^{(l)}=[(\omega_{1}^{(l)})^2\omega_{2}^{(l)}]^{1/3}$,
$c_{2}^{(l)}=[\omega_{1}^{(l)}(\omega_{2}^{(l)})^2]^{1/3}$,
$l=0,\ldots,5$.
\end{proof}

In order to prove the uniqueness of the limiting functions
$\widetilde{F}_{j}^{(i)}$, we need to use Lemma \ref{lemaaux}
below. More general versions of this result can be found in
\cite{AptLopRocha} (see Lemma 4.1) and \cite{Apt} (see Proposition
1.1), so we omit the proof.

Let us first introduce some notations. Assume that $\Delta_{1},
\Delta_{2}$ are disjoint compact intervals in $\mathbb{R}$, and
let $C(\Delta_{i})$ denote the space of real-valued continuous
functions on $\Delta_{i}$. We write
$\mathbf{u}=(u_{1},u_{2})^{t}\in C$ if $u_{1}\in C(\Delta_{2})$,
$u_{2}\in C(\Delta_{1})$. Given $u_{1}\in C(\Delta_{2})$, let
$T_{2,1}(u_{1})$ be the harmonic function in
$\overline{\mathbb{C}}\setminus\Delta_{2}$ that solves the
Dirichlet problem with boundary condition
\[
T_{2,1}(u_{1})(x)=u_{1}(x),\quad x\in\Delta_{2},
\]
and given $u_{2}\in C(\Delta_{1})$, let $T_{1,2}(u_{2})$ denote
the harmonic function in $\overline{\mathbb{C}}\setminus
\Delta_{1}$ that solves the Dirichlet problem with boundary
condition
\[
T_{1,2}(u_{2})(x)=u_{2}(x),\quad x\in\Delta_{1}.
\]
Consider the linear operator $T:C\longrightarrow C$ defined as
follows:
\[
T=\left[\begin{array}{cc}
0 & T_{1,2} \\

T_{2,1} & 0\\
\end{array} \right],
\]
and $I: C\longrightarrow C$ the identity operator. The auxiliary
result is the following
\begin{lem}\label{lemaaux}
If $\mathbf{u}\in C$ and $(2I-T)(\mathbf{u})=\mathbf{0}$, then
$\mathbf{u}=\mathbf{0}$.
\end{lem}

Now we prove that the limiting functions do not depend on the
sequence $\Lambda\subset\mathbb{N}$ for which
(\ref{defntildeF1i})--(\ref{defntildeF2i}) hold.

\begin{prop}\label{propuniqueness}
The limiting functions $\widetilde{F}_{j}^{(i)}$ are unique for
every $j\in\{1,2\}$ and $i\in\{0,\ldots,5\}$.
\end{prop}

\begin{proof}
For each fixed $i\in\{0,\ldots,5\}$, by Proposition \ref{prop5}
the functions $\log|F_{1}^{(i)}|, \log|F_{2}^{(i)}|$ satisfy
\begin{equation}\label{eq:system}
\left\{
\begin{array}{lll}
2\log|F_{1}^{(i)}(\tau)|-\log|F_{2}^{(i)}(\tau)|=\log|f_{i}(\tau)|,\quad \tau\in(0,\alpha^{3}],\\
\\
-\log|F_{1}^{(i)}(\tau)|+2\log|F_{2}^{(i)}(\tau)|=\log|g_{i}(\tau)|,\quad \tau\in[-b^{3},-a^{3}],\\
\end{array}
\right.
\end{equation}
where $f_{i}(\tau), g_{i}(\tau)$ equal $1/\tau, 1,$ or $\tau$,
depending on the value of $i$. Assume that the functions
$\widetilde{G}_{1}^{(i)}, \widetilde{G}_{2}^{(i)}$ satisfy
\[
\lim_{k\in\Lambda'}\frac{P_{6k+i+1}(z)}{P_{6k+i}(z)}=\widetilde{G}_{1}^{(i)}(z),\qquad
z\in\mathbb{C}\setminus[0,\alpha^{3}],
\]
\[
\lim_{k\in\Lambda'}\frac{P_{6k+i+1,2}(z)}{P_{6k+i,2}(z)}=\widetilde{G}_{2}^{(i)}(z),\qquad
z\in\mathbb{C}\setminus[-a^{3},-b^{3}],
\]
for some other subsequence $\Lambda'\subset\mathbb{N}$. As before,
we can find positive constants $d_{1}^{(i)}, d_{2}^{(i)}$ so that
the functions $G_{j}^{(i)}:=d_{j}^{(i)}\widetilde{G}_{j}^{(i)}$
satisfy the same system (\ref{eq:system}). If we define the
functions
\[
u_{1}:=\log|F_{1}^{(i)}|-\log|G_{1}^{(i)}|,\qquad
u_{2}:=\log|F_{2}^{(i)}|-\log|G_{2}^{(i)}|,\qquad
\mathbf{u}=(u_{1},u_{2})^{t},
\]
observe that $u_{1}$ is harmonic in
$\overline{\mathbb{C}}\setminus[0,\alpha^{3}]$, $u_{2}$ is
harmonic in $\overline{\mathbb{C}}\setminus[-b^{3},-a^{3}]$, and
they are also bounded in the corresponding regions. Moreover,
\begin{equation}\label{eq:systemu1u2}
\left\{
\begin{array}{lll}
2 u_{1}(\tau)-u_{2}(\tau)=0,\quad \tau\in(0,\alpha^{3}],\\
\\
-u_{1}(\tau)+2 u_{2}(\tau)=0,\quad \tau\in[-b^{3},-a^{3}].\\
\end{array}
\right.
\end{equation}

Let $\Delta_{1}:=[0,\alpha^{3}]$, $\Delta_{2}:=[-b^{3},-a^{3}]$.
From (\ref{eq:systemu1u2}) and the (generalized) maximum-minimum
principle for harmonic functions, we obtain that $2
u_{1}-T_{1,2}(u_{2})\equiv 0$ on
$\overline{\mathbb{C}}\setminus\Delta_{1}$ and $2
u_{2}-T_{2,1}(u_{1})\equiv 0$ on
$\overline{\mathbb{C}}\setminus\Delta_{2}$. In particular,
$(2I-T)(\mathbf{u})=\mathbf{0}$, so by Lemma \ref{lemaaux} we get
$u_{1}=0$ on $\Delta_{2}$ and $u_{2}=0$ on $\Delta_{1}$. Therefore
$T_{1,2}(u_{2})\equiv 0$ on $\overline{\mathbb{C}}\setminus
\Delta_{1}$ and $T_{2,1}(u_{1})\equiv 0$ on
$\overline{\mathbb{C}}\setminus\Delta_{2}$, implying that
$u_{1}\equiv 0$ and $u_{2}\equiv 0$. From
$|F_{j}^{(i)}|=|G_{j}^{(i)}|$ it easily follows that
$c_{j}^{i}=d_{j}^{i}$ and $\widetilde{F}_{j}^{(i)}=
\widetilde{G}_{j}^{(i)}$.
\end{proof}

\noindent{\bf Proof of Theorem \ref{introd:ratioasymptotics}.} The
existence of the limits
(\ref{eq:ratioasymp1})--(\ref{eq:ratioasymp2}) follows from the
normality of the families (\ref{familiesrational}) and Proposition
\ref{propuniqueness}. The polynomials $P_{n}$ satisfy:
\[
P_{3k}(z)=P_{3k+1}(z)+a_{3k}P_{3k-2}(z),\qquad\,\,
P_{3k+1}(z)=P_{3k+2}(z)+a_{3k+1}P_{3k-1}(z),\qquad\,\,
zP_{3k+2}(z)=P_{3k+3}(z)+a_{3k+2}P_{3k}(z),
\]
and so (\ref{eq:ratioasymp1}) implies that the following limits
hold:
\begin{equation}\label{eq:lima1}
\lim_{k\rightarrow\infty}a_{6k+i}=\widetilde{F}_{1}^{(i-2)}(z)\widetilde{F}_{1}^{(i-1)}(z)
(1-\widetilde{F}_{1}^{(i)}(z)),\qquad i\in\{0,1,3,4\},
\end{equation}
\begin{equation}\label{eq:lima2}
\lim_{k\rightarrow\infty}a_{6k+i}=\widetilde{F}_{1}^{(i-2)}(z)\widetilde{F}_{1}^{(i-1)}(z)
(z-\widetilde{F}_{1}^{(i)}(z)),\qquad i\in\{2,5\},
\end{equation}
where these relations are valid for every
$z\in\mathbb{C}\setminus[0,\alpha^{3}]$
($\widetilde{F}_{1}^{(-2)}=\widetilde{F}_{1}^{(4)}$,
$\widetilde{F}_{1}^{(-1)}=\widetilde{F}_{1}^{(5)}$).

We have:
\[
\widetilde{F}_{1}^{(i-2)}(z)\widetilde{F}_{1}^{(i-1)}(z)
(1-\widetilde{F}_{1}^{(i)}(z))=-C_{1}^{(i)}+O(1/z),\qquad
z\rightarrow\infty,\quad i\in\{0,1,3,4\},
\]
\[
\widetilde{F}_{1}^{(i-2)}(z)\widetilde{F}_{1}^{(i-1)}(z)
(z-\widetilde{F}_{1}^{(i)}(z))=-C_{0}^{(i)}+O(1/z),\qquad
z\rightarrow\infty,\quad i\in\{2,5\},
\]
and so (\ref{ratiocoeff}) follows from
(\ref{eq:lima1})--(\ref{eq:lima2}).
(\ref{eq:ratioQ1})--(\ref{eq:ratioQ3}) is a direct consequence of
(\ref{eq:ratioasymp1})--(\ref{eq:ratioasymp2}). \hfill $\Box$

\begin{prop}\label{propratioasymporthonPsi}
Assume that the hypotheses of Theorem
$\ref{introd:ratioasymptotics}$ hold. Then the polynomials $p_{n},
p_{n,2}$ defined in $(\ref{eq:definitionpn})$ satisfy for each
$i\in\{0,\ldots,5\}$$:$
\begin{equation}\label{eq:ratioasymporthon1}
\lim_{k\rightarrow\infty}\frac{p_{6k+i+1}(z)}{p_{6k+i}(z)}
=\kappa_{1}^{(i)}\widetilde{F}_{1}^{(i)}(z),\qquad
z\in\mathbb{C}\setminus[0,\alpha^{3}],
\end{equation}
\begin{equation}\label{eq:ratioasymporthon2}
\lim_{k\rightarrow\infty}\frac{p_{6k+i+1,2}(z)}{p_{6k+i,2}(z)}
=\kappa_{2}^{(i)}\widetilde{F}_{2}^{(i)}(z),\qquad
z\in\mathbb{C}\setminus[-b^3,-a^{3}],
\end{equation}
uniformly on compact subsets of the indicated regions, where
\[
\kappa_{j}^{(i)}=\sqrt{\omega_{j}^{(i)}},\qquad j=1,2,
\]
and the constants $\omega_{j}^{(i)}$ are defined in
$(\ref{defnomega1l})$--$(\ref{defnomega2lmasmas})$. Consequently,
for the leading coefficients $\kappa_{n}, \kappa_{n,2}$ defined in
$(\ref{defnkappankappan2})$ we have:
\begin{equation}\label{eq:ratiokappa1}
\lim_{k\rightarrow\infty}\frac{\kappa_{6k+i+1}}{\kappa_{6k+i}}=\kappa_{1}^{(i)},
\end{equation}
\begin{equation}\label{eq:ratiokappa2}
\lim_{k\rightarrow\infty}\frac{\kappa_{6k+i+1,2}}{\kappa_{6k+i,2}}=\kappa_{2}^{(i)}.
\end{equation}
In addition, the following limits hold uniformly on compact
subsets of $\mathbb{C}\setminus(S_{0}\cup S_{1})$$:$
\begin{equation}\label{ratioPsi1}
\lim_{k\rightarrow\infty}\frac{\Psi_{6k+i+1}(z)}{\Psi_{6k+i}(z)}=\frac{1}{\omega_{1}^{(i)}}
\frac{\widetilde{F}_{2}^{(i)}(z^3)}{z^2\,\widetilde{F}_{1}^{(i)}(z^3)},\qquad
i=0,3,
\end{equation}
\begin{equation}\label{ratioPsi2}
\lim_{k\rightarrow\infty}\frac{\Psi_{6k+i+1}(z)}{\Psi_{6k+i}(z)}=\frac{1}{\omega_{1}^{(i)}}
\frac{z\widetilde{F}_{2}^{(i)}(z^3)}{\widetilde{F}_{1}^{(i)}(z^3)},\qquad
i=1,2,4,5.
\end{equation}
\end{prop}
\begin{proof}
Using the same argument employed before and Theorems $1$ and $2$
from \cite{BCL}, we obtain
\[
\lim_{k\rightarrow\infty}\frac{p_{6k+i+1}(z)}{p_{6k+i}(z)}=S_{1}^{(i)}(z),\qquad
z\in\mathbb{C}\setminus[0,\alpha^3],\quad i=0,1,3,4,
\]
\[
\lim_{k\rightarrow\infty}\frac{p_{6k+i+1}(z)}{p_{6k+i}(z)}=S_{1}^{(i)}(z)\,\phi_{1}(z),\qquad
z\in\mathbb{C}\setminus[0,\alpha^3],\quad i=2,5,
\]
\[
\lim_{k\rightarrow\infty}\frac{p_{6k+i+1,2}(z)}{p_{6k+i,2}(z)}=S_{2}^{(i)}(z),\qquad
z\in\mathbb{C}\setminus[-b^3,-a^3],\quad i=0,1,2,
\]
\[
\lim_{k\rightarrow\infty}\frac{p_{6k+i+1,2}(z)}{p_{6k+i,2}(z)}=S_{2}^{(i)}(z)\,\phi_{2}(z),\qquad
z\in\mathbb{C}\setminus[-b^3,-a^3],\quad i=3,5,
\]
\[
\lim_{k\rightarrow\infty}\frac{p_{6k+5,2}(z)}{p_{6k+4,2}(z)}=(S_{2}^{(4)}(z)\,\phi_{2}(z))^{-1},\qquad
z\in\mathbb{C}\setminus[-b^3,-a^3],
\]
so (\ref{eq:ratioasymporthon1}) and (\ref{eq:ratioasymporthon2})
follow. (\ref{eq:ratiokappa1})--(\ref{eq:ratiokappa2}) are
immediate consequences of
(\ref{eq:ratioasymporthon1})--(\ref{eq:ratioasymporthon2}).

Observe that by (\ref{defnHn1}) we can write
\[
\frac{\Psi_{n+1}}{\Psi_{n}}=\frac{\kappa_{n}^2}{\kappa_{n+1}^2}
\frac{h_{n+1}}{h_{n}}\frac{Q_{n}}{Q_{n+1}}\frac{Q_{n+1,2}}{Q_{n,2}},
\]
so (\ref{eq:ratiokappa1})--(\ref{eq:ratiokappa2}) together with
Lemma \ref{lemmaweakstar} and Theorem
\ref{introd:ratioasymptotics} imply
(\ref{ratioPsi1})--(\ref{ratioPsi2}).
\end{proof}

\noindent{\bf Proof of Proposition \ref{introd:relation}.} We
first show that $a^{(i)}>0$ for all $i$. If $a^{(0)}=0$, then
(\ref{eq:lima1}) implies $\widetilde{F}_{1}^{(0)}\equiv 1$, and
using (\ref{eq:boundaryvalue1}) we obtain that
$\widetilde{F}_{2}^{(0)}(z)=z$ on
$\mathbb{C}\setminus[-b^{3},-a^{3}]$, contradicting
(\ref{LaurentexpansionFinfinity}). If $a^{(1)}=0$, then again by
(\ref{eq:lima1}) we get $\widetilde{F}_{1}^{(1)}\equiv 1$, and so
by (\ref{eq:boundaryvalue2}) we have
$\widetilde{F}_{2}^{(1)}\equiv 1$, contradicting
(\ref{eq:boundaryvalue5}). If $a^{(2)}=0$, then from
(\ref{eq:lima2}) it follows that $\widetilde{F}_{1}^{(2)}(z)=z$ on
$\mathbb{C}\setminus[0,\alpha^{3}]$, and so
(\ref{eq:boundaryvalue3}) implies that
$\widetilde{F}_{2}^{(1)}(z)=z$, which is impossible. Similar
arguments show that $a^{(i)}>0$ for $i\in\{3,4,5\}$.

Now we prove simultaneously that
$\widetilde{F}_{1}^{(2)}(z)=z\,\widetilde{F}_{1}^{(0)}(z)$ and
$\widetilde{F}_{2}^{(0)}=\widetilde{F}_{2}^{(2)}$. Let
\[
u_{1}(z):=\log|F_{1}^{(2)}(z)|-\log|z\,F_{1}^{(0)}(z)|,\qquad
u_{2}(z):=\log|F_{2}^{(2)}(z)|-\log|F_{2}^{(0)}(z)|.
\]
Then $u_{1}$ is harmonic in
$\overline{\mathbb{C}}\setminus[0,\alpha^{3}]$ and $u_{2}$ is
harmonic in $\overline{\mathbb{C}}\setminus[-b^{3},-a^{3}]$. By
(\ref{eq:boundaryvalue4}) and (\ref{eq:boundaryvalue6}) we see
that $u_{2}$ is bounded on
$\overline{\mathbb{C}}\setminus[-b^{3},-a^{3}]$. Taking into
account the definitions of the functions $S_{1}^{(0)}$ and
$S_{1}^{(2)}$, the boundedness of $u_{1}$ is equivalent to the
boundedness of the expression
\[
\frac{1}{2\pi}\int_{0}^{2\pi}
\Re\Big[\frac{e^{i\theta}+1/\phi_{1}(z)}{e^{i\theta}-1/\phi_{1}(z)}\Big]
\log(1+\cos\theta)\,d\theta-\log|z|,\qquad z\notin[0,\alpha^{3}],
\]
which follows trivially from the identity
\[
\frac{1}{2\pi}\int_{0}^{2\pi}
\Re\Big[\frac{e^{i\theta}+w}{e^{i\theta}-w}\Big]
\log|1+e^{i\theta}|\,d\theta=\log|1+w|,\qquad |w|<1.
\]

Now Proposition \ref{prop5} implies that $2
u_{1}(\tau)-u_{2}(\tau)=0$ for $\tau\in(0,\alpha^{3}]$, and
$-u_{1}(\tau)+2 u_{2}(\tau)=0$ for $\tau\in[-b^{3},-a^{3}]$. As in
the proof of Proposition \ref{propuniqueness}, this yields
$u_{1}\equiv 0$, $u_{2}\equiv 0$. Similarly one proves that
$\widetilde{F}_{1}^{(5)}(z)=z\widetilde{F}_{1}^{(3)}(z)$ and
$\widetilde{F}_{2}^{(5)}=\widetilde{F}_{2}^{(3)}$.

From (\ref{eq:relfunc11}), (\ref{ratiocoeff}), and
(\ref{Laurentexpansion}), it follows that $a^{(0)}=a^{(2)}$ and
$a^{(3)}=a^{(5)}$. We have by (\ref{eq:lima1})--(\ref{eq:lima2})
that
\[
\widetilde{F}_{1}^{(0)}(z)\widetilde{F}_{1}^{(1)}(z)
(z-\widetilde{F}_{1}^{(2)})=a^{(2)},\qquad\qquad
\widetilde{F}_{1}^{(4)}(z)\widetilde{F}_{1}^{(5)}(z)
(1-\widetilde{F}_{1}^{(0)})=a^{(0)}.
\]
Since $a^{(0)}=a^{(2)}$ and
$\widetilde{F}_{1}^{(2)}(z)=z\widetilde{F}_{1}^{(0)}(z)$, we
deduce that $z\widetilde{F}_{1}^{(0)}\widetilde{F}_{1}^{(1)}
=\widetilde{F}_{1}^{(4)}\widetilde{F}_{1}^{(5)}$, or equivalently
$\widetilde{F}_{1}^{(1)}\widetilde{F}_{1}^{(2)}
=\widetilde{F}_{1}^{(4)}\widetilde{F}_{1}^{(5)}$. The other two
relations in (\ref{eq:relfunc12}) follow immediately using this
equality and (\ref{eq:relfunc11}).

The relations in (\ref{eq:relfunc22}) are an easy consequence of
(\ref{eq:relfunc12}) and
(\ref{eq:boundaryvalue1})--(\ref{eq:boundaryvalue3}). Now,
(\ref{eq:relfunc13}) is obtained by dividing appropriate relations
from (\ref{eq:lima1})--(\ref{eq:lima2}), one by another, and
taking into account (\ref{eq:relfunc12}). The equality
$a^{(0)}+a^{(1)}=a^{(3)}+a^{(4)}$ follows immediately from
$\widetilde{F}_{1}^{(0)}\widetilde{F}_{1}^{(1)}=
\widetilde{F}_{1}^{(3)}\widetilde{F}_{1}^{(4)}$.

We next show that the functions $\widetilde{F}_{1}^{(i)}$,
$i\in\{0,\ldots,5\}$, are all distinct. If $i\in\{0,1,3,4\}$, then
evidently $\widetilde{F}_{1}^{(i)}\neq \widetilde{F}_{1}^{(2)}$
and $\widetilde{F}_{1}^{(i)}\neq \widetilde{F}_{1}^{(5)}$. If
$\widetilde{F}_{1}^{(0)}=\widetilde{F}_{1}^{(1)}$, then
(\ref{eq:boundaryvaluetilde1}) and (\ref{eq:boundaryvaluetilde2})
imply that
\[
\frac{\widetilde{F}_{2}^{(1)}(\tau)}{\widetilde{F}_{2}^{(0)}(\tau)}=\frac{\omega_{1}^{(1)}}
{\omega_{1}^{(0)}}\frac{1}{\tau},\qquad \tau\in(0,\alpha^{3}],
\]
which is contradictory since
$\widetilde{F}_{2}^{(1)}/\widetilde{F}_{2}^{(0)}$ is holomorphic
outside $[-b^{3},-a^{3}]$. The same argument proves that
$\widetilde{F}_{1}^{(0)}\neq\widetilde{F}_{1}^{(4)},
\widetilde{F}_{1}^{(1)}\neq\widetilde{F}_{1}^{(3)}$, and
$\widetilde{F}_{1}^{(3)}\neq\widetilde{F}_{1}^{(4)}$. If
$\widetilde{F}_{1}^{(0)}=\widetilde{F}_{1}^{(3)}$, then
(\ref{eq:boundaryvaluetilde1}) implies that
$\widetilde{F}_{2}^{(0)}=\widetilde{F}_{2}^{(3)}$, which is
impossible (cf. (\ref{LaurentexpansionFinfinity})). Similarly
(using now (\ref{eq:boundaryvaluetilde2}) and
(\ref{eq:boundaryvaluetilde3})) we see that
$\widetilde{F}_{1}^{(1)}\neq \widetilde{F}_{1}^{(4)}$ and
$\widetilde{F}_{1}^{(2)}\neq \widetilde{F}_{1}^{(5)}$.

Now we show that the functions $\widetilde{F}_{2}^{(i)}$,
$i\in\{0,1,3,4\}$, are all different. If we assume that
$\widetilde{F}_{2}^{(0)}=\widetilde{F}_{2}^{(1)}$, then
(\ref{eq:boundaryvaluetilde4})--(\ref{eq:boundaryvaluetilde5})
imply that
\[
\frac{|\widetilde{F}_{1}^{(1)}(\tau)|}{|\widetilde{F}_{1}^{(0)}(\tau)|}
=\frac{\omega^{(1)}_{2}}{\omega_{2}^{(0)}}\,\tau^{2},\qquad
\tau\in[-b^{3},-a^{3}].
\]
It follows that
$\widetilde{F}_{1}^{(1)}(z)=z^2\widetilde{F}_{1}^{(0)}(z)$, which
is impossible. The other cases are justified just by looking at
the Laurent expansion at infinity.

By (\ref{eq:relfunc13}) we see that $a^{(0)}\neq a^{(3)}$ and
$a^{(1)}\neq a^{(4)}$. Now we show that $a^{(1)}\neq a^{(3)}$.
Applying (\ref{eq:lima1}) for $i=0$ and the relation
$\widetilde{F}_{1}^{(1)}\widetilde{F}_{1}^{(2)}
=\widetilde{F}_{1}^{(4)}\widetilde{F}_{1}^{(5)}$, we get
\[
\widetilde{F}_{1}^{(1)}\widetilde{F}_{1}^{(2)}(1-\widetilde{F}_{1}^{(0)})=a^{(0)}.
\]
From this relation and (\ref{eq:lima1}) (for $i=4$), we obtain
\[
\widetilde{F}_{1}^{(1)}(1-\widetilde{F}_{1}^{(0)})
=\frac{a^{(0)}}{a^{(4)}}\,\widetilde{F}_{1}^{(3)}(1-\widetilde{F}_{1}^{(4)}).
\]
Applying the first two equations from (\ref{eq:relfunc13}), we
derive that
\begin{equation}\label{eq:aux15}
\widetilde{F}_{1}^{(1)}(1-\widetilde{F}_{1}^{(0)})
=\frac{a^{(3)}}{a^{(1)}}\,(1-\widetilde{F}_{1}^{(1)})
(\widetilde{F}_{1}^{(0)}-1)+\frac{a^{(0)}}{a^{(1)}}\,(1-\widetilde{F}_{1}^{(1)}).
\end{equation}
If we assume now that $a^{(1)}=a^{(3)}$, then (\ref{eq:aux15})
yields
$(1-\widetilde{F}_{1}^{(0)})/(1-\widetilde{F}_{1}^{(1)})=a^{(0)}/a^{(1)}$.
But from (\ref{eq:lima1}) we know that
\[
\frac{(1-\widetilde{F}_{1}^{(0)})\widetilde{F}_{1}^{(4)}}
{(1-\widetilde{F}_{1}^{(1)})\widetilde{F}_{1}^{(0)}}=\frac{a^{(0)}}{a^{(1)}},
\]
hence $\widetilde{F}_{1}^{(4)}=\widetilde{F}_{1}^{(0)}$, which is
contradictory. Therefore $a^{(1)}\neq a^{(3)}$, and so by
(\ref{eq:rellimitcoeff}) we also obtain that $a^{(0)}\neq
a^{(4)}$. \hfill $\Box$

\begin{co}\label{relationomegas}
The following relations hold:
\[
\omega_{1}^{(0)}\omega_{1}^{(1)}
=\omega_{1}^{(3)}\omega_{1}^{(4)},\qquad\omega_{1}^{(0)}=\omega_{1}^{(2)},\qquad
\omega_{1}^{(3)}=\omega_{1}^{(5)},
\]
\[
\omega_{2}^{(0)}\omega_{2}^{(1)}
=\omega_{2}^{(3)}\omega_{2}^{(4)},\qquad
\omega_{2}^{(0)}=\omega_{2}^{(2)},\qquad
\omega_{2}^{(3)}=\omega_{2}^{(5)}.
\]
\end{co}
\begin{proof}
All these relations follow immediately from the relations
established in Proposition \ref{introd:relation} and the boundary
value equations
(\ref{eq:boundaryvaluetilde1})--(\ref{eq:boundaryvaluetilde3}) and
(\ref{eq:boundaryvaluetilde4})--(\ref{eq:boundaryvaluetilde6})
(multiply or divide appropriately these equations, one by
another).
\end{proof}

\section{The Riemann surface representation of the
limiting functions
$\widetilde{F}_{j}^{(i)}$}\label{Riemannsurfacesection}

We will give now the proof of Theorem
\ref{theoRiemannsurfacerepresentation}. Before doing so, we need
some definitions and comments. Let
\[
G_{1}^{(i,j)}:=\frac{F_{1}^{(i)}}{F_{1}^{(j)}},\qquad
G_{2}^{(i,j)}:=\frac{F_{2}^{(i)}}{F_{2}^{(j)}},\qquad 0\leq
i,j\leq 5.
\]

Recall that the conformal representation $\psi$ of $\mathcal{R}$
onto $\overline{\mathbb{C}}$ satisfies
(\ref{eq:divisorcond1})--(\ref{eq:divisorcond2}). As a
consequence, we have $\psi(z)=\overline{\psi(\overline{z})}$. This
property implies in particular that
\[
\psi_{k}:\overline{\mathbb{R}}\setminus(\Delta_{k}\cup\Delta_{k+1})
\longrightarrow\overline{\mathbb{R}},\qquad
k=0,1,2,\qquad\Delta_{0}=\Delta_{3}=\emptyset,
\]
and
\begin{equation}\label{eq:symmetrybranches}
\psi_{k}(x_{\pm})=\overline{\psi_{k}(x_{\mp})}
=\overline{\psi_{k+1}(x_{\pm})},\qquad x\in\Delta_{k+1}.
\end{equation}
So all the coefficients in the Laurent expansion at infinity of
the branches $\psi_{k}$ are real. Given a function $F$ that
satisfies
\[
F(z)=C\,z^{k}+O(z^{k-1}),\qquad C\in\mathbb{R}\setminus\{0\},\quad
z\rightarrow\infty,
\]
we use the symbol $\sign(F(\infty))$ to denote the sign of $C$
(i.e. $\sign(F(\infty))=1$ if $C>0$ and $\sign(F(\infty))=-1$ if
$C<0$).

The function $\psi_{0}\,\psi_{1}\,\psi_{2}$ is analytic and
bounded on $\overline{\mathbb{C}}$, so this function is constant.
Let us denote this constant by $C$ (we will reserve in this
section the letter $C$ for this constant). So we have
\begin{equation}\label{eq:productobranches}
(\psi_{0}\,\psi_{1}\,\psi_{2})(z)\equiv C,\qquad
(\widetilde{\psi}_{0}\,\widetilde{\psi}_{1}\,\widetilde{\psi}_{2})(z)\equiv
1,\qquad z\in\overline{\mathbb{C}}.
\end{equation}

\begin{prop}\label{proprelationGRiemann}
The following relations hold:
\begin{equation}\label{eq:repG1}
G_{1}^{(0,3)}(z)=\frac{\sign((\psi_{1}\psi_{2})(\infty))
\,(\psi_{1}\psi_{2})(z)}{|C|^{2/3}},\qquad\qquad\,\,\,
G_{2}^{(0,3)}(z)=\frac{\sign(\psi_{2}(\infty))\,\psi_{2}(z)}{|C|^{1/3}}.
\end{equation}
\begin{proof}
By (\ref{eq:boundaryvalue1}) and (\ref{eq:boundaryvalue4}) we have
\begin{equation}\label{eq:auxRiemann1}
|G_{1}^{(0,3)}(\tau)|^{2}\frac{1}{G_{2}^{(0,3)}(\tau)}=1,\qquad
\tau\in(0,\alpha^3],
\end{equation}
\begin{equation}\label{eq:auxRiemann2}
|G_{2}^{(0,3)}(\tau)|^{2}\frac{1}{|G_{1}^{(0,3)}(\tau)|}=1,\qquad
\tau\in[-b^3,-a^3].
\end{equation}
Observe also that $G_{1}^{(0,3)}$ and $G_{2}^{(0,3)}$ are bounded
on $\overline{\mathbb{C}}\setminus\Delta_{1}$ and
$\overline{\mathbb{C}}\setminus\Delta_{2}$, respectively. Let us
call $v_{1}$ and $v_{2}$ the functions on the right hand side of
the relations (\ref{eq:repG1}), respectively. The function $v_{2}$
is positive on $\Delta_{1}=[0,\alpha^3]$ since
$\sign(v_{2}(\infty))=1$. Using
(\ref{eq:symmetrybranches})--(\ref{eq:productobranches}), for any
$x\in(0,\alpha^3)$,
\[
\frac{|v_{1}(x_{\pm})|^2}{v_{2}(x)}
=\frac{|\psi_{1}(x_{\pm})|^2\,\psi_{2}(x)^2}
{\sign(\psi_{2}(\infty))\,\psi_{2}(x)\,|C|}=\frac{|\psi_{0}(x_{\mp})||\psi_{1}(x_{\pm})||\psi_{2}(x)|}{|C|}
=\frac{|\overline{\psi_{0}(x_{\pm})}||\psi_{1}(x_{\pm})||\psi_{2}(x)|}
{|C|}=1,
\]
i.e. $v_{1}$ and $v_{2}$ satisfy (\ref{eq:auxRiemann1}) on
$(0,\alpha^3)$. On the other hand, for $x\in(-b^3,-a^3)$,
\[
\frac{|v_{2}(x_{\pm})|^2}{|v_{1}(x)|}=\frac{|\psi_{2}(x_{\pm})|}
{|\psi_{1}(x_{\pm})|}=1,
\]
so $v_{1}$ and $v_{2}$ also satisfy $(\ref{eq:auxRiemann2})$ on
$(-b^{3},-a^3)$. Finally, the same argument used to prove
Proposition \ref{propuniqueness} yields the validity of
(\ref{eq:repG1}).
\end{proof}
\end{prop}

\noindent\textbf{Proof of Theorem
\ref{theoRiemannsurfacerepresentation}.} By Proposition
\ref{proprelationGRiemann} we have:
\begin{equation}\label{relaFPsi1}
\frac{\widetilde{F}_{1}^{(4)}}{\widetilde{F}_{1}^{(1)}}=\frac{\widetilde{F}_{1}^{(0)}}{\widetilde{F}_{1}^{(3)}}
=\widetilde{\psi}_{1}\,\widetilde{\psi}_{2}=1/\widetilde{\psi}_{0},
\end{equation}
\begin{equation}\label{relaFPsi2}
\frac{\widetilde{F}_{2}^{(0)}}{\widetilde{F}_{2}^{(3)}}
=\widetilde{\psi}_{2}.
\end{equation}
From the first relation in (\ref{eq:relfunc13}) and
(\ref{relaFPsi1}), simple algebraic manipulations show that
\[
\widetilde{F}_{1}^{(0)}
=\frac{a^{(0)}-a^{(3)}}{a^{(0)}\widetilde{\psi}_{0}-a^{(3)}},\qquad
\widetilde{F}_{1}^{(3)}=\frac{(a^{(0)}-a^{(3)})
\,\widetilde{\psi}_{0}}{a^{(0)}\widetilde{\psi}_{0}-a^{(3)}}.
\]
The representations of $\widetilde{F}_{1}^{(2)}$ and
$\widetilde{F}_{1}^{(5)}$ follow immediately from the relations
$\widetilde{F}_{1}^{(2)}(z)=z \widetilde{F}_{1}^{(0)}(z)$ and
$\widetilde{F}_{1}^{(5)}(z)=z \widetilde{F}_{1}^{(3)}(z)$. The
relation
$\widetilde{F}_{1}^{(1)}/\widetilde{F}_{1}^{(4)}=\widetilde{\psi}_{0}$
and (\ref{eq:relfunc13}) prove the representations of
$\widetilde{F}_{1}^{(1)}$ and $\widetilde{F}_{1}^{(4)}$.

Recall that
\begin{equation}\label{recurrencePsin}
z\Psi_{n}(z)=\Psi_{n+1}+a_{n}\Psi_{n-2}, \qquad n\geq 2.
\end{equation}
Therefore, if we define the functions
\[
U^{(i)}(z):=\lim_{k\rightarrow\infty}\frac{\Psi_{6k+i+1}(z)}
{\Psi_{6k+i}(z)},\qquad z\in\mathbb{C}\setminus(S_{0}\cup S_{1}),
\qquad 0\leq i\leq 5,
\]
(by Proposition \ref{propratioasymporthonPsi} we know that such
limits exist) then we know by (\ref{recurrencePsin}) that
\[
a^{(i)}=U^{(i-2)}(z)U^{(i-1)}(z)(z-U^{(i)}(z)),\qquad 0\leq i\leq
5,
\]
where we understand that $U^{(-2)}=U^{(4)}, U^{(-1)}=U^{(5)}$. In
particular, applying (\ref{ratioPsi1}) and (\ref{ratioPsi2}) we
obtain for $i=0,1,4,5,$
\begin{equation}\label{eq:rela0F1F2}
a^{(0)}=\frac{1}{\omega_{1}^{(4)}\omega_{1}^{(5)}}
\frac{\widetilde{F}_{2}^{(5)}(z)}
{\widetilde{F}_{1}^{(5)}(z)}\frac{\widetilde{F}_{2}^{(4)}(z)}
{\widetilde{F}_{1}^{(4)}(z)}
\Big(z-\frac{\widetilde{F}_{2}^{(0)}(z)}{\omega_{1}^{(0)}
\widetilde{F}_{1}^{(0)}(z)}\Big),
\end{equation}
\begin{equation}\label{eq:rela1F1F2}
a^{(1)}=\frac{1}{\omega_{1}^{(0)}\omega_{1}^{(5)}}
\frac{\widetilde{F}_{2}^{(0)}(z)}
{\widetilde{F}_{1}^{(0)}(z)}\frac{\widetilde{F}_{2}^{(5)}(z)}
{\widetilde{F}_{1}^{(5)}(z)}
\Big(1-\frac{\widetilde{F}_{2}^{(1)}(z)}{\omega_{1}^{(1)}
\widetilde{F}_{1}^{(1)}(z)}\Big),
\end{equation}
\begin{equation}\label{eq:rela4F1F2}
a^{(4)}=\frac{1}{\omega_{1}^{(2)}\omega_{1}^{(3)}}
\frac{\widetilde{F}_{2}^{(2)}(z)}
{\widetilde{F}_{1}^{(2)}(z)}\frac{\widetilde{F}_{2}^{(3)}(z)}
{\widetilde{F}_{1}^{(3)}(z)}
\Big(1-\frac{\widetilde{F}_{2}^{(4)}(z)}{\omega_{1}^{(4)}
\widetilde{F}_{1}^{(4)}(z)}\Big),
\end{equation}
\begin{equation}\label{eq:rela5F1F2}
a^{(5)}=\frac{1}{\omega_{1}^{(3)}\omega_{1}^{(4)}}
\frac{\widetilde{F}_{2}^{(3)}(z)}
{\widetilde{F}_{1}^{(3)}(z)}\frac{\widetilde{F}_{2}^{(4)}(z)}
{\widetilde{F}_{1}^{(4)}(z)}
\Big(z-\frac{\widetilde{F}_{2}^{(5)}(z)}{\omega_{1}^{(5)}
\widetilde{F}_{1}^{(5)}(z)}\Big),
\end{equation}
where these identities are valid for every
$z\in\mathbb{C}\setminus([-b^{3},-a^3]\cup[0,\alpha^3])$. If we
apply the relations $a^{(3)}=a^{(5)}$,
$\widetilde{F}_{1}^{(5)}=z\widetilde{F}_{1}^{(3)}$,
$\widetilde{F}_{2}^{(5)}=\widetilde{F}_{2}^{(3)}$, from
(\ref{eq:rela0F1F2}) and (\ref{eq:rela5F1F2}) we obtain
\[
z\frac{a^{(0)}}{a^{(3)}}\Big(1-\frac{1}{\omega_{1}^{(5)}}
\frac{\widetilde{F}_{2}^{(3)}(z)}{\widetilde{F}_{1}^{(5)}(z)}\Big)
=\frac{\omega_{1}^{(3)}}{\omega_{1}^{(5)}}\Big(z-
\frac{\widetilde{F}_{2}^{(0)}(z)}{\omega_{1}^{(0)}
\widetilde{F}_{1}^{(0)}(z)}\Big).
\]
Using (\ref{relaFPsi2}) and substituting in this expression the
functions $\widetilde{F}_{1}^{(0)}$ and $\widetilde{F}_{1}^{(5)}$
by their representations in terms of the branches
$\widetilde{\psi}_{k}$, we get
\[
z\Big(\frac{a^{(0)}}{a^{(3)}}
-\frac{\omega_{1}^{(3)}}{\omega_{1}^{(5)}}\Big)
=\frac{(a^{(0)}\widetilde{\psi}_{0}(z)-a^{(3)})}{(a^{(0)}-a^{(3)})}
\Big(\frac{a^{(0)}}{a^{(3)}\widetilde{\psi}_{0}(z)}
-\frac{\omega_{1}^{(3)}\widetilde{\psi}_{2}(z)}{\omega_{1}^{(0)}}\Big)
\frac{\widetilde{F}_{2}^{(3)}(z)}{\omega_{1}^{(5)}}.
\]
The factors in the right hand side of this equation never vanish
on $\mathbb{C}\setminus([0,\alpha^{3}]\cup[-b^3,-a^3])$, and so we
can write
\[\widetilde{F}_{2}^{(3)}(z)=\frac{z\big(\frac{a^{(0)}}{a^{(3)}}
-\frac{\omega_{1}^{(3)}}{\omega_{1}^{(5)}}\big)\omega_{1}^{(5)}
(a^{(0)}-a^{(3)})}{(a^{(0)}\widetilde{\psi}_{0}(z)-a^{(3)})
\big(\frac{a^{(0)}}{a^{(3)}\widetilde{\psi}_{0}(z)}
-\frac{\omega_{1}^{(3)}\widetilde{\psi}_{2}(z)}{\omega_{1}^{(0)}}\big)}.
\]
If we move $z$ to the left hand side and evaluate at infinity we
obtain
\begin{equation}\label{relomegasanda1}
\omega_{1}^{(5)}\Big(\frac{a^{(0)}}{a^{(3)}}-
\frac{\omega_{1}^{(3)}}{\omega_{1}^{(5)}}\Big)=\frac{a^{(0)}}{a^{(3)}},
\end{equation}
and so the Riemann surface representation for
$\widetilde{F}_{2}^{(3)}$ follows. This also proves the
representation for the functions $\widetilde{F}_{2}^{(5)},
\widetilde{F}_{2}^{(0)}$, and $\widetilde{F}_{2}^{(2)}$.

From (\ref{eq:rela1F1F2}) and (\ref{eq:rela4F1F2}) we derive the
relation
\[
\frac{a^{(1)}}{a^{(4)}}\Big(1- \frac{\widetilde{F}_{2}^{(4)}}
{\omega_{1}^{(4)}\widetilde{F}_{1}^{(4)}}\Big)
=\frac{\omega_{1}^{(2)}\omega_{1}^{(3)}}
{\omega_{1}^{(0)}\omega_{1}^{(5)}}
\Big(1-\frac{\widetilde{F}_{2}^{(1)}}{\omega_{1}^{(1)}
\widetilde{F}_{1}^{(1)}}\Big).
\]
From Corollary \ref{relationomegas} we know that
$\omega_{1}^{(2)}\omega_{1}^{(3)}=\omega_{1}^{(5)}\omega_{1}^{(0)}$.
Since $\widetilde{F}_{2}^{(4)}/\widetilde{F}_{2}^{(1)}
=\widetilde{F}_{2}^{(0)}/\widetilde{F}_{2}^{(3)}=\widetilde{\psi}_{2}$
and
$\widetilde{F}_{1}^{(4)}/\widetilde{F}_{1}^{(1)}=1/\widetilde{\psi}_{0}
=\widetilde{\psi}_{1}\widetilde{\psi}_{2}$, we get
\[
\frac{a^{(1)}}{a^{(4)}}-1
=\frac{\widetilde{F}_{2}^{(4)}}{\widetilde{F}_{1}^{(4)}}
\Big(\frac{a^{(1)}}{a^{(4)}\omega_{1}^{(4)}}
-\frac{\widetilde{\psi}_{1}}{\omega_{1}^{(1)}}\Big).
\]
Evaluating at infinity we obtain the relation
\begin{equation}\label{relomegasanda2}
\omega_{1}^{(1)}=\frac{a^{(4)}}{a^{(4)}-a^{(1)}},
\end{equation}
and so we can write
\[
\widetilde{F}_{2}^{(4)}=\frac{\widetilde{F}_{1}^{(4)}}
{(\widetilde{\psi}_{1}-(\omega_{1}^{(1)}-1)/\omega_{1}^{(4)})}.
\]
Therefore, the Riemann surface representation of
$\widetilde{F}_{2}^{(4)}$ follows from that of
$\widetilde{F}_{1}^{(4)}$ and the representation of
$\widetilde{F}_{2}^{(1)}$ follows from the relation
$\widetilde{F}_{2}^{(4)}=\widetilde{\psi}_{2}\widetilde{F}_{2}^{(1)}$.

Now from $(\ref{relomegasanda1})$ and Corollary
\ref{relationomegas} we get
\begin{equation}\label{eq:relationomega3aes}
\omega_{1}^{(3)}=\omega_{1}^{(5)}=\frac{a^{(0)}}{a^{(0)}-a^{(3)}}.
\end{equation}
If we evaluate both sides of the equation (\ref{eq:rela5F1F2}) at
infinity we obtain
\[
a^{(5)}=a^{(3)}
=\frac{1}{\omega_{1}^{(3)}\omega_{1}^{(4)}}(1-1/\omega_{1}^{(3)}),
\]
and so (\ref{eq:relationomega3aes}) gives
$\omega_{1}^{(4)}=(a^{(0)}-a^{(3)})/(a^{(0)})^2$. Finally, from
Corollary \ref{relationomegas} and the above computations we
deduce that
$\omega_{1}^{(0)}=\omega_{1}^{(2)}=(a^{(4)}-a^{(1)})/(a^{(0)}a^{(4)})$.
\hfill$\Box$

\begin{rmk}\label{remarksobreaes}
Since $\omega_{1}^{(1)}>0$, it follows from
$(\ref{relomegasanda2})$ that $a^{(4)}>a^{(1)}$.
\end{rmk}

\noindent\textbf{Proof of Proposition \ref{algebraicequation}.} It
is straightforward to check that the function
\[
\chi(z)=\psi\Big(-\frac{a^3}{2}(1+z)\Big)-\psi(\infty^{(0)}),\qquad
\infty^{(0)}\in\mathcal{R},
\]
is a conformal representation of the Riemann surface $\mathcal{S}$
constructed as $\mathcal{R}$ (\ref{definitionRiemannsurfaceR}) but
formed by the sheets
\[
\mathcal{S}_{0}:=\overline{\mathbb{C}}\setminus[-\mu,-1], \qquad
\mathcal{S}_{1}:=\overline{\mathbb{C}}\setminus([-\mu,-1]\cup[1,\lambda]),
\qquad \mathcal{S}_{2}:=\overline{\mathbb{C}}\setminus[1,\lambda],
\]
where $\lambda$ and $\mu$ are defined in (\ref{defnlambdamu}).
$\chi$ also satisfies
\[
\chi(z)=z+O(1),\qquad z\rightarrow\infty^{(1)},
\]
and has a simple zero at $\infty^{(0)}\in\mathcal{S}$. Observe
that $\chi(\infty^{(2)})=-\psi(\infty^{(0)})$ (the reader is
cautioned that in this relation, $\infty^{(2)}\in\mathcal{S}$ and
$\infty^{(0)}\in\mathcal{R}$).

$\chi$ and $\mathcal{S}$ are the type of conformal mappings and
Riemann surfaces analyzed in \cite{LPRY}. It follows from
\cite[Theorem 3.1]{LPRY} that $\chi(\infty^{(2)})=2/H(\beta)$,
where $H$ and $\beta$ are described in the statement of
Proposition \ref{algebraicequation} (the uniqueness of $\beta$ and
$\gamma$ is justified in \cite{LPRY}). So
$\chi(z)=\psi(-a^3(1+z)/2)+2/H(\beta)$. It also follows from
\cite[Theorem 3.1]{LPRY} that the function $w=H(\beta)\chi(z)-1$
is the solution of the algebraic equation
\[
w^3-(H(\beta)z+\Theta_{1}-\Theta_{2}-h)w^2
-(1+\Theta_{1}+\Theta_2)w+H(\beta)z-h=0,
\]
where $\Theta_{1}, \Theta_{2},$ and $h$ are the constants
described in the statement of Proposition \ref{algebraicequation}.
Simple computations and a change of variable yield immediately
that $w=\psi(z)$ is the solution of the equation
(\ref{algebraicequationenunciado}). \hfill$\Box$

\section{The $n$th root asymptotics and zero
asymptotic distribution of the polynomials $Q_{n}$ and
$Q_{n,2}$}\label{sectionnthroot}

It is well-known (see \cite{SaffTotik}) that if
$E\subset\mathbb{C}$ is a compact set that is regular with respect
to the Dirichlet problem, and $\phi$ is a continuous real-valued
function on $E$, then there exists a unique $\widetilde{\mu}\in
\mathcal{M}_1(E)$ satisfying the variational conditions
\[
V^{\widetilde{\mu}}(z)+\phi(z) \left\{ \begin{array}{l} = w,\quad
z
\in \supp{(\widetilde{\mu})}, \\
\geq w, \quad z\in E,
\end{array} \right.
\]
for some constant $w$. The measure $\widetilde{\mu}$ is called the
equilibrium measure in the presence of the external field $\phi$
on $E$, and $w$ the equilibrium constant.

Recall that we defined $\lambda_{1}$ to be the positive,
rotationally invariant measure on $S_{0}$ whose restriction to the
interval $[0,\alpha]$ coincides with the measure $s_{1}(x)\,dx$,
and we defined $\lambda_{2}$ to be the positive, rotationally
invariant measure on $S_{1}$ whose restriction to the interval
$[-b,-a]$ coincides with the measure $s_{2}(x)\,dx$.

\begin{lem}\label{lemmaauxregular}
Suppose that $\lambda_{1}, \lambda_{2} \in \Reg$. Then the
measures
\begin{equation}\label{eq:primerasmedidas}
\frac{s_{1}(\sqrt[3]{\tau})}{\tau^{2/3}}\,d\tau,\qquad
s_{1}(\sqrt[3]{\tau})\,\sqrt[3]{\tau}\,d\tau,\qquad
\tau\in[0,\alpha^{3}],
\end{equation}
\begin{equation}\label{eq:segundasmedidas}
s_{2}(\sqrt[3]{\tau})\,d\tau,
\quad\frac{s_{2}(\sqrt[3]{\tau})}{\sqrt[3]{\tau}}\,d\tau,\quad
\frac{s_{2}(\sqrt[3]{\tau})}{\tau^{2/3}}\,d\tau,\quad
\tau\in[-b^3,-a^3],
\end{equation}
are also regular.
\end{lem}
\begin{proof}
Let $\pi_{n}$ be the $n$th monic orthogonal polynomial associated
with $\lambda_{1}$, i.e. $\pi_{n}$ is the monic polynomial of
degree $n$ that satisfies
\begin{equation}\label{eq:defnortpin}
\int_{S_{0}}\pi_{n}(t)\,\overline{t^{k}}\,d\lambda_{1}(t)=0,\qquad
0\leq k\leq n-1.
\end{equation}
It is immediate to check that
\[\pi_{n}(e^{\frac{2\pi
i}{3}}z)=e^{\frac{2\pi i n}{3}}\pi_{n}(z).
\]
We deduce from this property and (\ref{eq:defnortpin}) that the
polynomials
$$
\pi_{3k}(\sqrt[3]{\tau}),\qquad\frac{\pi_{3k+1}(\sqrt[3]{\tau})}{\sqrt[3]{\tau}},\qquad
\frac{\pi_{3k+2}(\sqrt[3]{\tau})}{\tau^{2/3}},
$$
are precisely the monic orthogonal polynomials of degree $k$
associated, respectively, with the measures
\begin{equation}\label{tresmedidas}
\frac{s_{1}(\sqrt[3]{\tau})}{\tau^{2/3}}\,d\tau,\qquad
s_{1}(\sqrt[3]{\tau})\,d\tau,\qquad
s_{1}(\sqrt[3]{\tau})\,\tau^{2/3}\,d\tau.
\end{equation}
We also have:
\[
\int_{S_{0}}|\pi_{3k}(t)|^2\,d\lambda_{1}(t)=\int_{0}^{\alpha^3}(\pi_{3k}(\sqrt[3]{\tau}))^2\,
\frac{s_{1}(\sqrt[3]{\tau})}{\tau^{2/3}}\,d\tau,
\]
\[
\int_{S_{0}}|\pi_{3k+1}(t)|^2\,d\lambda_{1}(t)=\int_{0}^{\alpha^3}
\Big(\frac{\pi_{3k+1}(\sqrt[3]{\tau})}{\sqrt[3]{\tau}}\Big)^2\,s_{1}(\sqrt[3]{\tau})\,d\tau,
\]
\[
\int_{S_{0}}|\pi_{3k+2}(t)|^2\,d\lambda_{1}(t)=\int_{0}^{\alpha^3}
\Big(\frac{\pi_{3k+2}(\sqrt[3]{\tau})}{\tau^{2/3}}\Big)^2\,s_{1}(\sqrt[3]{\tau})\,\tau^{2/3}\,d\tau.
\]
So taking into account (see \cite[Theorem 5.2.5]{Ransford}) that
$$
\capp(\supp(\lambda_{1}))=\capp(\supp(\rho))^{1/3},
$$
where $\capp(A)$ denotes the logarithmic capacity of $A$, and
$\rho$ is any of the three measures in (\ref{tresmedidas}), the
regularity of $\lambda_{1}$ implies the regularity of the three
measures in (\ref{tresmedidas}).

Let $l_{n}$ denote the $n$th monic orthogonal polynomial
associated with the measure
$d\rho_{1}(\tau):=s_{1}(\sqrt[3]{\tau})\,\sqrt[3]{\tau}\,d\tau$,
and let $T_{n}$ be the $n$th Chebyshev polynomial (see
\cite{Ransford}, page 155) for the set $E:=\supp{(\rho_{1})}$. We
have
\[
\Big(\int l_{n}^2(\tau)\,d\rho_{1}(\tau)\Big)^{1/2}\leq \Big(\int
T_{n}^{2}(\tau)\,d\rho_{1}(\tau)\Big)^{1/2} \leq
\|T_{n}\|_{E}\,\rho_{1}(E)^{1/2},
\]
where $\|T_{n}\|_{E}$ denotes the supremum norm of $T_{n}$ on $E$,
and so by \cite[Corollary 5.5.5]{Ransford} we obtain
\begin{equation}\label{eq:upperboundreg}
\limsup_{n\rightarrow\infty}\|l_{n}\|_{2}^{1/n}\leq
\lim_{n\rightarrow\infty}\|T_{n}\|_{E}^{1/n}=\capp(\supp(\rho_{1})).
\end{equation}
If we call $\widetilde{l}_{n}$ the $n$th monic orthogonal
polynomial associated with the measure
$d\rho_{2}(\tau):=s_{1}(\sqrt[3]{\tau})\,\tau^{2/3}\,d\tau$, we
have
\[
\Big(\int \widetilde{l}_{n}^2(\tau)d\rho_{2}(\tau)\Big)^{1/2}\leq
\alpha^{1/2}\Big(\int l_{n}^{2}(\tau)\,d\rho_{1}(\tau)\Big)^{1/2},
\]
and so the regularity of $\rho_{2}$ and (\ref{eq:upperboundreg})
imply the regularity of $\rho_{1}$. Similar arguments show that
the measures in (\ref{eq:segundasmedidas}) are regular.
\end{proof}

\noindent\textbf{Proof of Theorem \ref{theodist}.} Recall that if
$P$ is a polynomial, we indicate by $\mu_{P}$ the associated
normalized zero counting measure. Let $j\in\{0,\ldots,5\}$ be
fixed, and assume that for some subsequence
$\Lambda\subset\mathbb{N}$ we have:
\[
\mu_{P_{6k+j}}\stackrel{*}{\longrightarrow}\mu_{1}\in\mathcal{M}_{1}(\Delta_{1}),\qquad\qquad
\mu_{P_{6k+j,2}}\stackrel{*}{\longrightarrow}\mu_{2}\in\mathcal{M}_{1}(\Delta_{2}).
\]
Consequently,
\begin{equation}\label{eq:conshyp1}
\lim_{k\in\Lambda}\frac{1}{2k}\log|P_{6k+j}(z)|=-V^{\mu_{1}}(z),\qquad
z\in\mathbb{C}\setminus\Delta_{1},
\end{equation}
\begin{equation}\label{eq:conshyp2}
\lim_{k\in\Lambda}\frac{1}{4k}\log|P_{6k+j,2}(z)|=-\frac{1}{4}V^{\mu_{2}}(z),\qquad
z\in\mathbb{C}\setminus\Delta_{2},
\end{equation}
uniformly on compact subsets of the indicated regions.

We know by Proposition \ref{prop4} that there exists a fixed
measure $d\rho$ supported on $\Delta_{1}$ ($d\rho$ is one of the
measures in (\ref{eq:primerasmedidas})) such that
\begin{equation}\label{eq:orthogonality1}
0=\int_{\Delta_{1}}\tau^{j}\,P_{6k+j}(\tau)\,\frac{d\rho(\tau)}{P_{6k+j,2}(\tau)},\qquad
0\leq j<\deg(P_{6k+j}).
\end{equation}
We know by Lemma \ref{lemmaauxregular} that the measure $d\rho$ is
regular. If we apply \cite[Lemma 4.2]{FLLS} (taking, in the
notation of \cite{FLLS}, $d\sigma=d\rho$, $\phi_{2k}=1/P_{6k+j,2}$
and $\phi=-(1/4)V^{\mu_{2}}$), we obtain from (\ref{eq:conshyp2})
and (\ref{eq:orthogonality1}) that $\mu_{1}$ is the equilibrium
measure in the presence of the external field
$\phi=-(1/4)V^{\mu_{2}}$, hence
\begin{equation}\label{eq:vectequilprob1}
V^{\mu_{1}}(\tau)-\frac{1}{4}\,V^{\mu_{2}}(\tau)\left\{
\begin{array}{lll}
=w_{1}, &  &\quad\tau\in\supp(\mu_{1}),\\
\\
\geq w_{1}, &  &\quad\tau\in\Delta_{1},
\end{array}
\right.
\end{equation}
and
\begin{equation}\label{eq:asymp1}
\lim_{k\in\Lambda}\Big(\int_{\Delta_{1}}P_{6k+j}^{2}(\tau)\,d\nu_{6k+j}(\tau)\Big)^{1/4k}
=e^{-w_{1}},
\end{equation}
where the measure $d\nu_{6k+j}$ is defined in
(\ref{orthogvaryingmeasures}).

By Proposition \ref{prop3}, there exists a fixed measure $d\eta$
($d\eta$ is one of the measures in (\ref{eq:segundasmedidas}))
supported on $\Delta_{2}$ such that
\begin{equation}\label{eq:orthogonality2}
0=\int_{\Delta_{2}}\tau^{j}\,P_{6k+j,2}(\tau)\,\frac{|h_{6k+j}(\sqrt[3]{\tau})|}{|P_{6k+j}(\tau)|}\,d\eta(\tau),\qquad
0\leq j<\deg(P_{6k+j,2}).
\end{equation}
The function $h_{6k+j}$ is defined in (\ref{eq:definitionhn}). We
also know by Lemma \ref{lemmaauxregular} that $d\eta$ is regular.
Taking into account the representations $(\ref{intreph3k})$ and
the fact that $p_{n}$ is orthonormal with respect to $d\nu_{n}$
(see (\ref{eq:definitionpn}) and Proposition
\ref{orthonormalitypn}), it follows that there exist positive
constants $C_{1}, C_{2}$ such that
\[
C_{1}\leq |h_{6k+j}(\sqrt[3]{\tau})|\leq C_{2}\qquad \mbox{for
all}\quad \tau\in\Delta_{2}.
\]
So applying again \cite[Lemma 4.2]{FLLS} (now take
$d\sigma=d\eta$,
$\phi_{k}(\tau)=|h_{6k+j}(\sqrt[3]{\tau})|/|P_{6k+j}(\tau)|$ and
$\phi=-V^{\mu_{1}}$), we get from (\ref{eq:orthogonality2}) and
(\ref{eq:conshyp1}) that $\mu_{2}$ is the equilibrium measure in
the presence of the external field $\phi=-V^{\mu_{1}}$, and so
\begin{equation}\label{eq:vectequilprob2}
V^{\mu_{2}}(\tau)-V^{\mu_{1}}(\tau)\left\{
\begin{array}{lll}
=w_{2}, &  &\quad\tau\in\supp(\mu_{2}),\\
\\
\geq w_{2}, &  &\quad\tau\in\Delta_{2},
\end{array}
\right.
\end{equation}
and
\begin{equation}\label{eq:asymp2}
\lim_{k\in\Lambda}\Big(\int_{\Delta_{2}}P_{6k+j,2}^{2}(\tau)\,d\nu_{6k+j,2}(\tau)\Big)^{1/2k}
=e^{-w_{2}},
\end{equation}
where the measure $d\nu_{6k+j,2}$ is defined in
(\ref{orthogvaryingmeasures}).

By (\ref{eq:vectequilprob1}) and (\ref{eq:vectequilprob2}), the
vector measure $(\mu_{1},\mu_{2})$ solves the equilibrium problem
determined by the interaction matrix (\ref{defninteractionmatrix})
on the intervals $\Delta_{1}, \Delta_{2}$. Since the solution to
this equilibrium problem must be unique, (\ref{weakstar1})
follows. (\ref{eq:asymp1}) and (\ref{eq:asymp2}) imply
(\ref{asympK1})--(\ref{asympK2}). Finally,
(\ref{eq:nthrootasympPn})--(\ref{eq:nthrootasympPn2}) are an
immediate consequence of (\ref{weakstar1}).\hfill $\Box$

\vspace{0.2cm}

\noindent{\bf Proof of Proposition
\ref{relratioandnthrootasymptotics}.} By Theorem
\ref{introd:ratioasymptotics} we know that the following limit
holds:
\[
\lim_{k\rightarrow\infty}\frac{Q_{6(k+1)}(z)}{Q_{6k}(z)}
=\prod_{i=0}^{5}\widetilde{F}_{1}^{(i)}(z^3),\qquad
z\in\mathbb{C}\setminus S_{0}.
\]
Therefore we obtain that
\[
\lim_{k\rightarrow\infty}|Q_{6k}(z)|^{1/k}
=\prod_{i=0}^{5}|\widetilde{F}_{1}^{(i)}(z^3)|,\qquad
z\in\mathbb{C}\setminus S_{0},
\]
and by Corollary \ref{cortheodist} it follows that
\[
e^{-\frac{1}{3}V^{\overline{\mu}_{1}}(z^3)}
=\prod_{i=0}^{5}|\widetilde{F}_{1}^{(i)}(z^3)|^{1/6}\qquad
z\in\mathbb{C}\setminus S_{0}.
\]
So (\ref{relrationthroot1}) is proved. The same argument proves
(\ref{relrationthroot2}). \hfill $\Box$

\vspace{0.2cm}

\noindent{\bf Acknowledgments.} The results of this paper are part
of my Ph.D. dissertation at Vanderbilt University. I am thankful
to my advisor Edward B. Saff for the many useful discussions we
had concerning this work. This research was partially supported by
the U.S. National Science Foundation grant DMS-0808093. I am also
thankful to the Belgian Interuniversity Attraction Pole (grant
P06/02) and the Fonds voor Wetenschappelijk Onderzoek (FWO) for
financing my postdoctoral studies at Katholieke Universiteit
Leuven, where the writing of this paper was completed.












\end{document}